\documentclass[preprint,10pt]{amsart}
\usepackage{amsmath,amssymb,amsthm,mathrsfs,verbatim}
\usepackage[margin=1.0in]{geometry}

\usepackage{amsmath}
\usepackage{amssymb}
\usepackage{amsfonts}
\usepackage{geometry}
\usepackage{url}

\usepackage{amsmath}
\usepackage{amssymb}
\usepackage{amsfonts}
\usepackage{geometry}
\usepackage{url}
\usepackage{setspace}
\usepackage{amsthm}
\usepackage[all,cmtip]{xy}
\usepackage{amsmath,amscd}

\makeatletter
\def\url@leostyle{%
  \@ifundefined{selectfont}{\def\UrlFont{\sf}}{\def\UrlFont{\small\ttfamily}}}
\makeatother
\DeclareMathOperator{\Ker}{Ker}

\newtheorem{thm}{Theorem}[section]
\newtheorem{prop}[thm]{Proposition}
\newtheorem{lemma}[thm]{Lemma}
\newtheorem{cor}[thm]{Corollary}
\newtheorem{conj}[thm]{Conjecture}

\newtheorem{quest}[thm]{Question}

\theoremstyle{definition}
\newtheorem{defn}[thm]{Definition}

\newtheorem{rem}[thm]{Remark}
\newtheorem{example}[thm]{Example}

\newcommand{\Hom}{\text{Hom}}
\newcommand{\Frac}{\text{Frac}}

\newcommand{\gr}{\text{gr}}

\newcommand{\rad}{\text{rad}}
\newcommand{\ldot}{\textbf{.}}

\newcommand{\triv}{\mathrm{triv}}
\newcommand{\rank}{\text{rank}}
\newcommand{\reg}{\text{reg}}
\newcommand{\loc}{\text{loc}}

\newcommand{\sgn}{\text{sgn}}

\newcommand{\mf}{\mathfrak}

\newcommand{\Z}{\mathbb{Z}}

\newcommand{\F}{\mathbb{F}}
\newcommand{\h}{\mf{h}}

\begin{document}

%\begin{frontmatter}

\title{Representations of Rational Cherednik Algebras in Positive Characteristic}
\author{Martina Balagovi\' c, Harrison Chen}
%\cortext[cor1]{corresponding author}
\address{
M.B: Department of Mathematics, University of York, York, YO10 5DD, UK, and \\
Department of Mathematics, University of Zagreb, Bijeni\v{c}ka 30, 10000 Zagreb, Croatia\\
H.C: Department of Mathematics, University of California, 852 Evans Hall, Berkeley, CA 94720 USA
 %Department of Mathematics,  Massachusetts Institute of Technology, Cambridge, MA 02139, USA 
}
\email{martina.balagovic@york.ac.uk, chenhi@math.berkeley.edu}

\begin{abstract}
 We study rational Cherednik algebras over an algebraically closed field of positive characteristic. 
We first prove several general results about category $\mathcal{O}$, and then  focus on rational Cherednik algebras associated to the general and special linear group over a finite field of the same characteristic as the underlying algebraically closed field. For such algebras we calculate the characters of irreducible representations with trivial lowest weight.
\end{abstract}

%\begin{keyword} rational Cherednik algebra \sep field of finite characteristic \sep irreducible representation \sep character

%\MSC 17B10 \sep 16W99 \sep 14G17

%\end{keyword}

%\end{frontmatter}

\maketitle

\section{Introduction}

Given an algebraically closed field $\Bbbk$, a finite-dimensional $\Bbbk$-vector space $\h$, a finite group $G\subseteq GL(\h)$ generated by reflections, a constant parameter $t\in \Bbbk$ and a collection of constants $c_{s} \in \Bbbk$ labeled by  conjugacy classes of reflections $s$ in $G$, the rational Cherednik algebra $H_{t,c}(G,\h)$ is a certain non-commutative infinite-dimensional associative algebra over $\Bbbk$, which deforms the semidirect product of the group algebra $\Bbbk G$ and the symmetric algebra on $\h$ and $\h^*$, $S(\h^{*} \oplus \h)$. Rational Cherednik algebras have been extensively studied since the early 1990s, and most efforts have focused on the case when the underlying field $\Bbbk$ is the field of complex numbers. This paper is one of the first attempts to study their representation theory in the case where $\Bbbk$ is an algebraically closed field of finite characteristic $p$. 

The parameter $t$ can be rescaled by a nonzero constant, producing two families of algebras with different types of behavior: one for  $t=0$ and one for $t\ne 0$, the latter being equivalent to $t=1$.

In characteristic zero, one commonly defines a category of $H_{t,c}(G,\h)$-representations called category $\mathcal{O}$, analogous to category $\mathcal{O}$ in Lie theory. It is generated (under taking subquotients and extensions) by standard or Verma modules $M_{t,c}(\tau)$, which are parametrized by irreducible representations $\tau$ of the finite group $G$. Verma modules admit a contravariant form $B$, such that the kernel of the form is the unique maximal proper submodule $J_{t,c}(\tau)$ of $M_{t,c}(\tau)$. The quotients $L_{t,c}(\tau) = M_{t,c}(\tau)/\Ker B$ comprise all irreducible modules in $\mathcal{O}$.

In positive characteristic we define category $\mathcal{O}=\mathcal{O}_{t,c}$ in a way that allows us to formulate and prove analogues of the properties and results in characteristic zero. One significant difference is that, while in characteristic zero and for generic choice of parameters Verma modules $M_{t,c}(\tau)$ are irreducible, this never happens in positive characteristic. The reason is that the algebra $H_{t,c}(G,\h)$ has a large center, so the module $M_{t,c}(\tau)$ always has a large submodule. To account for this, we define \emph{baby Verma modules} $N_{t,c}(\tau)$, which are quotients of Verma modules by the action of a certain large central subspace of $H_{t,c}(G,\h)$. Baby Verma modules are finite dimensional; consequently, all the irreducible quotients $L_{t,c}(\tau)$ are finite-dimensional, and we define category $\mathcal{O}$ to be the category of finite dimensional graded modules. It contains all the baby Verma modules, but not the Verma modules. This is analogous to the situation in the representation theory of Lie algebras in positive characteristic and to the study of the rational Cherednik algebras  $H_{0,c}(G,\h)$ in characteristic zero.

We prove the standard theorems about category $\mathcal{O}$. Namely, we show that any irreducible object, up to grading shifts, is isomorphic to some $L_{t,c}(\tau)$, and that the baby Verma modules admit a contravariant form with the usual properties. We define characters and show that for $t\ne 0$ and generic $c$, the characters of $L_{t,c}(\tau)$ are of a specific form, depending on the structure of a certain \emph{reduced module} $R_{t,c}(\tau)$. We calculate an upper bound for the dimension of irreducible modules. 

From here, we turn to investigate the characters of $L_{t,c}(\tau)$ for $t=0$ and $t=1$ and all values of $c$ for specific classes of groups $G$. Over a field $\Bbbk$ of characteristic $p$, we study the rational Cherednik algebra associated to the general and special linear group over a finite field, $G=GL_{n}(\F_{q})$ and $G=SL_{n}(\F_{q})$, for $q=p^r$. These groups are reflection groups which have no counterpart in characteristic zero. 

We describe the characters of the irreducible modules $L_{t,c}(\mathrm{triv})$ associated to the trivial representation of $GL_n(\F_q)$, for all $c$ (Theorems \ref{glnt0} and \ref{main}). One remarkable property is that for sufficiently large $n$ and $p$, the character of $L_{1,c}(\mathrm{triv})$ does not depend on the specific choice of $c$.  This phenomenon does not arise in characteristic zero. 

For $n,p,r$ large enough, the structure (specifically, the conjugacy classes) of $GL_n(\F_{q})$ and $SL_n(\F_{q})$ is similar enough to enable us to obtain the characters of the irreducible modules $L_{t,c}(\mathrm{triv})$ for $SL_n(\F_{q})$  as a corollary of the corresponding results for $GL_n(\F_{q})$. We specify the exact conditions on  $n,p,r$ being ``large enough" and calculate the remaining characters, thus obtaining the complete classification of $L_{t,c}(\mathrm{triv})$ for the rational Cherednik algebra associated to $SL_n(\F_{q})$. The main results about $SL_n(\F_{q})$ are Theorems \ref{slnt0} and \ref{sln}.

In the first phase of this work we have gathered information about the structure of irreducible representations through calculations of the contravariant form $B$ in the algebra software MAGMA \cite{magma}. We used this data to formulate conjectures and to settle the cases of small groups, which are often different from the general situation (for $GL_{n}(\F_{p^r})$, ``small" refers to small $n,p,r$).

This is the first paper in a series of two that we are planning on this topic. The second paper will be dedicated to studying, in more detail, category $\mathcal{O}$ for the rational Cherednik algebras associated to the group $GL_2(\F_p)$, and calculating the characters of irreducible representations $L_{t,c}(\tau)$ for all $\tau$.

The roadmap of this paper is as follows. Section 2 contains definitions of rational Cherednik algebras, Dunkl operators, baby Verma modules, category $\mathcal{O}$, and general results analogous to the ones in characteristic zero. In Section 3 we define characters of objects in category $\mathcal{O}$ and discuss characters of irreducible modules for generic $c$. In Section 4 we study characters of irreducible modules associated to the trivial lowest weight for $G=GL_n(\F_q)$, and offer the complete classification for all $t$ and $c$.  In Section 5 we solve the same problem for 
$G=SL_n(\F_q)$. The main results are Theorems \ref{glnt0}, \ref{main}, \ref{slnt0} and \ref{sln}.
Appendix \ref{secdata} contains data and conjectures about characters of rational Cherednik algebras associated to orthogonal groups over a finite field, which might be used for further work.

\section{Definitions and basic properties}\label{definitions}

\subsection{Notation}\label{notation}

Let $\Bbbk$ be an algebraically closed field of characteristic $p$, $\F_q$ a finite field of $q=p^r$ elements, $G$ a finite group, $\Bbbk G$ its group algebra, $\h$ a faithful $n$-dimensional representation of $G$ over $\Bbbk $, and $\h^*$ its dual representation. We may regard $G$ as a subgroup of $GL(\h)$. We will often let $\{x_1, \ldots, x_n\}$ denote a basis of $\h^*$ and $\{y_1, \ldots, y_n\}$ the dual basis of $\h$. Let $(\cdot,\cdot)$ be the canonical pairing $\h \otimes \h^* \rightarrow \Bbbk$ or $\h^* \otimes \h \rightarrow \Bbbk$.  

For any vector space $V$ let $TV$ and $SV$ denote the tensor and symmetric algebra of $V$ over $\Bbbk$, and $S^iV$ the homogeneous subspace of $SV$ of degree $i$.  For a graded vector space $M$, let $M_i$ denote the $i$-th graded piece, and $M[j]$ the same vector space with the grading shifted by $j$, meaning $M[j]_i = M_{i+j}$. For $M=\oplus M_i$ a graded vector space, define its Hilbert series as $$\mathrm{Hilb}_M(z)=\sum_{i} \dim M_i z^i.$$ For a filtered module $F$, let $\gr(F)$ denote the associated graded module. For an associative algebra $A$, $a,b \in A$, $S\subseteq A$, let $[a, b] = ab - ba$ be the usual commutator, and let $\langle S \rangle$ be the ideal generated by the subset $S$.

For $V$ some space of polynomials and $m=p^a$ a power of the characteristic $p$, we define $V^m$ to be the set $\{ f^m | f\in V \}$. If $V\subset \Bbbk [x_1,\ldots x_n]$ is graded and the Hilbert series of $\Bbbk [x_1,\ldots x_n]/V$ is $h(z)$, then the Hilbert series of $\Bbbk [x_1,\ldots x_n]/\left< V^m \right>\cong \Bbbk [x_1,\ldots x_n]\otimes_{\Bbbk [x_1^m,\ldots x_n^m]}\left( \Bbbk [x_1^m,\ldots x_n^m]/\left< V^m \right>\right)$ is
$h(z^m)\left( \frac{1-z^m}{1-z} \right)^n.$

For $\lambda \in \F_q, \lambda\ne 0$, let $d_{\lambda}$ be the following elements of $GL_{n}(\F_{q})$:
$$\textit{for } \lambda\ne 1,\,\,  d_\lambda = \left[ \begin{array}{ccccc} 
\lambda^{-1} & 0 & 0 & \cdots & 0\\
0&1&0 & \cdots &0\\
0&0&1 & \cdots &0\\
\vdots & \vdots & \vdots & \ddots & \vdots\\
0&0&0&\cdots & 1 \end{array}\right],\,\,\,\;
d_1 = \left[\begin{array}{ccccc} 
1 & 1 & 0 & \cdots & 0\\
0&1&0 & \cdots&0\\
0 & 0 & 1 & \cdots & 0\\
\vdots & \vdots & \vdots & \ddots & \vdots \\
0&0&0&\cdots & 1 \end{array}\right].$$

\subsection{Reflection groups}\label{reflections}

\begin{defn}
An element $s \in GL(\h)$ is a \emph{reflection} if the rank on $\h$ of $1 - s$ is 1. A finite subgroup $G \subset GL(\h)$ is a \emph{reflection group} if it is generated by reflections.
\end{defn}

\begin{example}
Let $\mathfrak{h}= \Bbbk ^n$ and $\F_q \subseteq \Bbbk $ a finite subfield. The group $GL_{n}(\F_{q})$ is a finite subgroup of $GL(\h)$, generated by conjugates of $d_{\lambda}$. All $d_{\lambda}$ and their conjugates are reflections, so $GL_{n}(\F_{q})$ is a reflection group.
\end{example}

This is the main example, in the sense that for $\Bbbk$ the algebraic closure of $\F_p$, and for any reflection group $G\subseteq GL(\h)$ one can construct the finite field of coefficients of $G$, which is some finite field of characteristic $p$, and view $G$ as a subgroup of $GL_{n}(\F_{q})$. Let $\h_{\F}=\F_q^n$ be the corresponding $\F_{q}$-form of $\h=\Bbbk^n$ preserved by $G$. 

Reflections in $G$ are elements that are conjugate in $GL(\h)$ to some $d_{\lambda}$. If $G=GL_{n}(\F_{q})$, there are $q$ conjugacy classes of reflections, with representatives $d_{\lambda}$. If $G$ is a proper subgroup of $GL_{n}(\F_{q})$, it is possible for elements of $G$ to be conjugate in $GL(\h)$ but not in $G$, and it is thus possible to have reflections which are not $G$-conjugate to any $d_{\lambda}$, and to have reflections with the same eigenvalues which are not in the same conjugacy class. 

A reflection $s\in G$ is called a \emph{semisimple reflection} if it is semisimple as an element of $GL(\h)$; such elements are conjugate in $GL(\h)$ to some $d_{\lambda}$ with $\lambda\ne 1$. A reflection $s\in G$ is called a \emph{unipotent reflection} if it is unipotent as an element of $GL(\h)$; such an element is conjugate in $GL(\h)$ to $d_{1}$, and has the property that $s^p = 1$. Note that in characteristic zero, a unipotent reflection generates an infinite subgroup, so considering unipotent reflections is unique to working in positive characteristic.

The group $G$ acts on $\mathfrak{h}$, and we define the structure of the dual representation on $\mathfrak{h}^*$ in the standard way: for $x\in \h^*, y\in \h, g\in G$, let $(g\ldot y,x)=(y,g^{-1}\ldot x)$. After choosing bases of $\mathfrak{h}$ and $\mathfrak{h}^*$, we see that $g\in GL_{n}(\F_{q})\subseteq GL(\h)$ acts in the dual representation as a matrix $(g^{-1})^{t}\in GL_{n}(\F_{q})\subseteq GL(\h^*)$. Consequently, $g$ is a reflection on $\h$ with eigenvalue $\lambda^{-1}$ if and only if it is a reflection on $\h^*$ with eigenvalue $\lambda$.

\begin{prop}\label{reflequiv} There exists a bijection between the set of reflections in $GL(\h_{\F})$ and the set of all vectors $\alpha \otimes \alpha^\vee \ne 0$ in $\h_{\F}^* \otimes \h_{\F}$ such that $(\alpha, \alpha^\vee) \ne 1$. The reflection $s$ corresponding to $\alpha \otimes \alpha^\vee$ acts:
\begin{align*}
\text{on $\h^*$ by} \quad & s \ldot x = x - (\alpha^\vee, x) \alpha \\
\text{on $\h$ by} \quad & s \ldot y = y + \frac{(y, \alpha)}{1 - (\alpha, \alpha^\vee)} \alpha^\vee.
\end{align*}

If $(\alpha, \alpha^\vee) \ne 0$, such a reflection $s$ is semisimple, acting on $\h^*$ with eigenvalue $1$ of multiplicity $n-1$ and eigenvalue $\lambda=1 - (\alpha^\vee, \alpha)\ne 1$ of multiplicity $1$. If $(\alpha, \alpha^\vee) = 0$, the reflection is unipotent, acting on $\h^*$ with eigenvalue $1$ of multiplicity $n-1$ and one Jordan block of size $2$. \end{prop}

\begin{proof}

For any reflection $s$, the operator $1-s$ on $\h_{\F}^*$ has rank one, and can thus be represented as $\alpha \otimes \alpha^\vee \in \h_{\F}^* \otimes \h_{\F}$, so that for all $x\in \h^*$, $$(1-s)\ldot x=(\alpha^\vee, x) \alpha.$$ Since $s$ is invertible, $0\ne s\ldot \alpha= (1-(\alpha^{\vee}, \alpha)) \alpha $, so $(\alpha^{\vee}, \alpha)\ne 1$ and $1-(\alpha^{\vee}, \alpha)$ is the eigenvalue of $s$ on $\h^*$ that is not $1$. The formula for the action of $s$ on $\h$ in terms of $\alpha \otimes \alpha^\vee $ follows from the one for the action of $s$ on $\h^*$ and the definition of dual representation. 

Conversely, given any  $\alpha \otimes \alpha^\vee \in \h_{\F}^* \otimes \h_{\F}$, the above formulas define dual reflections on $\h_{\F}$ and $\h_{\F}^*$.

Any vector $x$ in the kernel of  $\alpha^{\vee}$ is an eigenvector of $s$ with eigenvalue $1$. If $(\alpha, \alpha^\vee) \ne 0$, then $\alpha$ is an eigenvector with eigenvalue $1 - (\alpha^\vee, \alpha)\ne 1$, and $s$ is semisimple. If $(\alpha, \alpha^\vee) = 0$, then $s^i \ldot x = x - i(\alpha^\vee, x)\alpha$, so $s^p = 1$ and $s$ is unipotent.
\end{proof}

Let $S\subseteq G$ be the set of all reflections in $G$. Let $\alpha_{s}$ and $\alpha_{s}^{\vee}$ denote a choice of vectors from the proposition which correspond to a reflection $s$. Let $\lambda_{s}=1-(\alpha^\vee_s, \alpha_s)$; then $s$ has an eigenvalue $\lambda_s$ on $\h^*$ and $\lambda_s^{-1}$ on $\h$. In $GL(\mathfrak{h_{\F}})$, $s$ is conjugate to the $d_{\lambda_{s}^{-1}}$ from Subsection \ref{notation}.

%We will restrict our study to reflection groups, but all the results regarding Cherednik algebras of reflection groups can easily be extended to all finite groups, since the behavior of the Cherednik algebra depends only on the reflections in the group.

\subsection{Rational Cherednik algebras}\label{cherednik}

Let $t \in \Bbbk $.  Let $c: S \rightarrow \Bbbk $ be a conjugation invariant function on the set of reflections, which we write as $s\mapsto c_s$.  
\begin{defn}
The \emph{rational Cherednik algebra} $H_{t,c}(G, \h)$ is the quotient of the semidirect product $\Bbbk G \ltimes T(\h \oplus \h^*)$ by the ideal generated by relations:
$$[x, x'] = 0, \,\, [y, y'] = 0, \,\, [y, x] = (y, x)t - \sum_{s \in S} c_s ((1 - s)\ldot x, y) s,$$
for all $x,x'\in \h^*,y,y'\in \h.$
\end{defn}

Note that for $g \in G$ and $y \in \h$, we use notation  $gy$ for multiplication in the algebra, and  $g \ldot y$ for the action from the representation; they are related by $gyg^{-1} = g \ldot y$.

The parameters $t$ and $c$ can be simultaneously rescaled, in the sense that $H_{t,c}(G, \h)\cong H_{at,ac}(G, \h)$ for any $a\in \Bbbk ^{\times}$ (the isomorphism sends generators $x \in \h^*$ to $tx$ and fixes $\h$ and $G$). This implies that it is enough to study two cases with respect to $t$, $t=0$ and $t=1$. We are mostly interested in the $t=1$ case.

It is clear from the definition that any element of the algebra $H_{t,c}(G,\h)$ can be written as a linear combination of $g x_1^{a_1} \cdots x_n^{a_n}y_1^{b_1}\cdots y_n^{b_n}$ for some  $g \in G$ and some integers $ a_i,b_i\ge 0$. The common question asked for algebras defined in such a way by generators and relations is whether this set is also linearly independent. The answer in the case of Cherednik algebras is positive, as proven in \cite{griffeth}, Theorem 2.1.

\begin{thm}[PBW Theorem for Cherednik algebras]\label{pbw}
The set $$\{g x_1^{a_1} \cdots x_n^{a_n}y_1^{b_1}\cdots y_n^{b_n} \mid g \in G, a_i,b_i \geq 0\}$$ is a basis for $H_{t,c}$.
\end{thm}

We will need a localization lemma, which is proved as in characteristic zero (see \cite{etingof-ma}). Let $\h_\reg$ be the subspace of $\h$ consisting of elements that are not contained in any reflection hyperplane, meaning they are not fixed by any $s\in S$. Let $\mathcal{D}(\h_\reg)$ be the algebra of differential operators on $\h_{\reg}.$ Let $H_{t,c}^\loc(G,\h)$ be the localization of $H_{t,c}(G,\h)$ as $\Bbbk[\h]=S\h^*$-module away from the reflection hyperplanes. Define the map 
$\Theta_{t,c}: H_{t,c}(G,\h) \rightarrow \Bbbk G \ltimes \mathcal{D}(\h_\reg)$ by $\Theta_{t,c}(x)=x$, $\Theta_{t,c}(y)=t \partial_y - \sum_{s \in S} c_s \frac{(\alpha_s, y)}{\alpha_s} (1 - s)$, $\Theta_{t,c}(g)=g$ for $x\in \h^*, y\in \h, g\in G$.

\begin{lemma}[Localization lemma]\label{loclemma}
The induced map of  localizations $$\Theta_{t,c}: H_{t,c}^\loc \rightarrow \Bbbk G \ltimes \mathcal{D}(\h_\reg)$$ is an isomorphism of algebras.
\end{lemma}

\subsection{Verma Modules $M_{t,c}(\tau)$ and Dunkl operators} \label{vermaintro}

\begin{defn}
Let $\tau$ be an irreducible finite-dimensional representation of $G$. Define a $\Bbbk G \ltimes S\h$-module structure on it by requiring the $\h$-action on $\tau$ to be zero. The \emph{Verma module} is the induced $H_{t,c}(G, \h)$-module  $M_{t,c}(G, \h, \tau) = H_{t,c}(G, \h) \otimes_{\Bbbk G \ltimes S\h} \tau$. We will write $M_{t,c}(\tau)$  instead of $M_{t,c}(G, \h, \tau)$ when it is clear what are $G$ and $\h$.
\end{defn}

As it is an induced module, it satisfies the following universal mapping property:
\begin{lemma}[Frobenius reciprocity] \label{homexist}
Let $M$ be an $H_{t,c}(G, \h)$-module. Let $\tau \subset M$ be a $G$-submodule on which $\h\subseteq H_{t,c}(G, \h)$ acts as zero. Then there is a unique homomorphism $\phi: M_{t,c}(\tau) \rightarrow M$ of $H_{t,c}(G, \h)$-modules such that $\phi|_\tau$ is the identity.
\end{lemma}

Define a $\mathbb{Z}-$grading on $H_{t,c}(G,\h)$  by letting $x \in \h^*$ have degree 1, $y \in \h$ have degree -1, and $g \in G$ have degree 0. We will denote by the subscript $+$ the positive degree elements of a graded module.

By the PBW Theorem \ref{pbw}, $M_{t,c}(\tau) \cong S\h^* \otimes \tau$ as $\Bbbk$-vector spaces. The grading on $M_{t,c}(\tau)$ by degree of $S\h^*$ is compatible with the grading on $H_{t,c}(G,\h)$ defined above, and we consider it as a graded module.  

Through the identification $M_{t,c}(\tau) \cong S\h^* \otimes \tau$, the action of the generators of $H_{t,c}(G,\h)$ can be explicitly written as follows. Let $f \otimes v \in S\h^* \otimes \tau \cong M_{t,c}(\tau) $, $x \in \h^*$, $y\in \h$ and $g \in G$.  Then $$x \ldot (f \otimes v) = (xf) \otimes v,$$ $$g \ldot (f \otimes v) = g\ldot f \otimes g \ldot v$$ 
$$y \ldot (f\otimes v) = t\partial_y(f) \otimes v - \sum_{s \in S} c_s \frac{(y, \alpha_s)}{\alpha_s} (1 - s) \ldot f \otimes s \ldot v.$$ The operators on $M_{t,c}(\tau)$ corresponding to the action of $y\in \h$, given by $$D_{y}= t\partial_y \otimes 1 - \sum_{s \in S} c_s \frac{(y, \alpha_s)}{\alpha_s} (1 - s) \otimes s,$$ are called \emph{Dunkl operators}.

We say a homogeneous element $v\in M_{t,c}(\tau)$ is \emph{singular} if $D_yv=0$ for all $y\in \h$. Any such element of positive degree generates a proper $H_{t,c}(G,\h)$ submodule. By Lemma \ref{homexist}, this submodule is isomorphic to a quotient of $M_{t,c}((\Bbbk G)v)$.

\begin{rem}The more common definition of the Dunkl operator in characteristic zero is
$$D_y = t\partial_y \otimes 1 - \sum_{s \in S} c_s \frac{2(y, \alpha_s)}{(1 - \lambda_s)\alpha_s} (1 - s) \otimes s.$$ The difference is due to a different convention in normalization of $\alpha_{s}$ and $\alpha_{s}^{\vee}$. As explained in Proposition \ref{reflequiv}, for a given reflection $s$, the vectors $\alpha_{s}$ and $\alpha_{s}^{\vee}$ are determined uniquely up to multiplication by nonzero scalars. We choose their mutual normalization so that $(\alpha, \alpha^{\vee})=1-\lambda_{s}$, while the common normalization in characteristic zero is  $(\alpha, \alpha^{\vee})=2$. That brings about an additional factor of $\frac{2}{1-\lambda_{s}}$ in some formulas. The reason for choosing a non-standard convention is that in characteristic $p$ there are unipotent reflections, for which $\lambda_{s}=1$ and $(\alpha, \alpha^{\vee})=0$. 
\end{rem}

\subsection{Contravariant Form}\label{formdef}
The results from this section can be found in Section 3.11 of \cite{etingof-ma}.

There is an analogue of Shapovalov form on Verma modules. First, for any graded $H_{t,c}(G, \h)$-module $M$ with finite dimensional graded pieces, define its restricted dual $M^\dagger$ to be the graded module whose $i-$th graded piece it the dual of the $i-$th graded piece of $M$. It is a left module for the opposite algebra $H_{t,c}(G, \h)^{opp}$ of $H_{t,c}(G, \h)$. Let $\bar{c}:S\to \Bbbk$ be the function $\bar{c}(s) = c(s^{-1})$. There is a natural isomorphism $H_{t,c}(G, \h)^{opp} \rightarrow H_{t,\bar{c}}(G, \h^*)$ that is the identity on $\h$ and $\h^*$, and sends $g \mapsto g^{-1}$ for $g \in G$, making $M^\dagger$ a $H_{t,\bar{c}}(G, \h^*)$-module.

\begin{defn}
Let $\tau$ be an an irreducible finite-dimensional representation of $G$.  By Lemma \ref{homexist}, there is a unique homomorphism $\phi: M_{t,c}(G, \h, \tau) \rightarrow M_{t,\bar{c}}(G, \h^*, \tau^*)^\dagger$ which is the identity in the lowest graded piece $\tau$.  By adjointness, it is equivalent to the  \emph{contravariant form} pairing
$$B: M_{t,c}(G, \h, \tau) \times M_{t,\bar{c}}(G, \h^*, \tau^*) \rightarrow \Bbbk.$$
\end{defn}

\begin{prop}\label{bform}
The contravariant form $B$ satisfies the following properties. \begin{itemize}
\item[a)] It is $G$-invariant: for $f \in M_{t,c}(\tau)$, $h \in M_{t,\bar{c}}(\tau^*)$, $B(g \ldot f, g \ldot h) = B(f, h)$.
\item[b)] For $x \in \h^*$, $f \in M_{t,c}(\tau)$, and $h \in M_{t,\bar{c}}(\tau^*)$, $B(xf, h) = B(f, D_{x}(h))$.
\item[c)] For $y \in \h$, $f \in M_{t,c}(\tau)$, and $h \in M_{t,\bar{c}}(\tau^*)$, $B(f, yh) = B(D_{y}(f), h)$.
\item[d)] The form is zero on elements in different degrees: if $f \in M_{t,c}(\tau)_i$ and $h \in M_{t,\bar{c}}(\tau^*)_j$, $i \ne j$, then $B(f, h) = 0$.
\item[e)] The form is the canonical pairing of $\tau$ and $\tau^*$ in the zeroth degree: for $v \in \tau = M_{t,c}(\tau)_0$, $f \in \tau^* = M_{t,\bar{c}}(\tau^*)_0$, $B(v, f) = (v, f)$.\end{itemize}
\end{prop}

As $B$ respects the grading of $M_{t,c}(\tau)$ and $M_{t,\bar{c}}(\tau^*)$, we can think of it as a collection of bilinear forms on their finite-dimensional graded pieces. Let $B_i$ be the restriction of $B$ to the $M_{t,c}(\tau)_{i}\otimes M_{t,\bar{c}}(\tau^*)_{i}$. We define $\Ker B$ to be $\Ker \phi\subseteq M_{t,c}(G, \h, \tau)$. Singular vectors of positive degree in $M_{t,c}(\tau)$ are in $\Ker B$, and so are the submodules generated by them.

\subsection{Baby Verma modules $N_{t,c}(\tau)$ and irreducible modules $L_{t,c}(\tau)$}

\begin{defn}
For $\tau$ an irreducible $G$-representation, define the $H_{t,c}(G,\h)$-representation $L_{t,c}(\tau)=L_{t,c}(G, \h, \tau)$ as the quotient $M_{t,c}(G, \h, \tau)/\Ker B$.  
\end{defn}

The modules $L_{t,c}(\tau)$ are graded. We are going to show that they are irreducible, as in characteristic zero. A notable difference is that, while in characteristic zero, for generic $t$ and $c$, the module $M_{1,c}(\tau)$ is irreducible and hence equal to $L_{1,c}(\tau)$, in characteristic $p$ or for $t=0$ this never happens. On the contrary, all $L_{t,c}(\tau)$ are finite dimensional, and $M_{t,c}(\tau)$ always have a large submodule. Because of that, we sometimes prefer using \emph{baby Verma modules} defined below instead of Verma modules. The definition and the name are analogous to the ones used in Lie theory in characteristic $p$ and to the notation used in \cite{gordon} for rational Cherednik algebras in characteristic zero at $t=0$.

Let $((S\h^*)^G)_+$ denote the subspace of $G$-invariants in $S\h^*$ of positive degree.  At $t=1$, the subspace $((S\h^*)^G)^p_+$ is central in $H_{1,c}(G,\h)$, and $((S\h^*)^G)^p_+ M_{1,c}(\tau)$ is a proper submodule of $M_{1,c}(\tau)$. 

\begin{defn}
The \emph{baby Verma module} for the algebra $H_{1,c}(G,\h)$ is the quotient $$N_{1,c}(\tau)=N_{1,c}(G, \h, \tau) = M_{1,c}(\tau)/((S\h^*)^G)^p_+ M_{1,c}(\tau).$$
\end{defn}

Since $(S\h^*)^G$ is graded, $N_{1,c}(\tau)$ is a graded module. The subspace $((S\h^*)^G)^p_+ M_{1,c}(\tau)$ is contained in $\Ker B$. To see this, let $Z\in ((S\h^*)^G)^p_+ $ be an arbitrary  homogeneous element of positive degree $m$, $v\in \tau$ and $y\in \h$ arbitrary. Then $D_{y}(Z\otimes v)=(yZ)\ldot v=(Zy)\ldot v=Z\ldot (y\ldot v)=0$, so $Z\otimes v$ is singular and therefore in $\Ker B$.

Because of this, the form $B$  descends to the $N_{1,c}(\tau)$, and $L_{1,c}(\tau)$ can be alternatively realized as $N_{1,c}(\tau)/\Ker B$.

To define baby Verma modules at $t=0$, we use that $((S\h^*)^G)_+$ is central in $H_{0,c}(G,\h)$, so $((S\h^*)^G)_+ M_{0,c}(\tau)$ is a proper submodule of $M_{0,c}(\tau)$. 

\begin{defn}
The \emph{baby Verma module} for the algebra $H_{0,c}(G,\h)$ is the quotient $$N_{0,c}(\tau)=N_{0,c}(G, \h, \tau) = M_{0,c}(\tau)/((S\h^*)^G)_+ M_{0,c}(\tau).$$ 
\end{defn}

By the same arguments as above, it is graded, the form $B$ descends to it, and $L_{0,c}(\tau)$ can be alternatively realized as a quotient of $L_{0,c}(\tau)$ by the kernel of $B$.

%In particular, they are both graded. Their Hilbert series, $$\chi_{N_{1,c}(\tau)}(z)=\sum_i \dim N_{1,c}(\tau)_iz^i$$ and $$\chi_{N_{0,c}(\tau)}(z)=\sum_i \dim N_{0,c}(\tau)_iz^i$$ are related by $$\chi_{N_{1,c}(\tau)}(z)=\left( \frac{1-z^p}{1-z} \right)^n \chi_{N_{0,c}(\tau)}(z^p).$$

Next, we turn to the basic properties of modules $L_{t,c}(\tau)$ and $N_{t,c}(\tau)$. We will need the following lemma, which is a consequence of the Hilbert-Noether Theorem and can be found in \cite{smithl} as Corollary 2.3.2.
\begin{lemma}\label{hilbertnoether}
For any finite group $G$, field $F$, and a finite dimensional $F[G]$-module $\h$, the algebra of invariants $(S\h)^G$ is finitely generated over $F$, and $S\h$ is a finite integral extension of $(S\h)^G$.
\end{lemma}

The following proposition is unique to fields of positive characteristic. 
\begin{prop}\label{lc-findim}
All $N_{t,c}(\tau)$, and thus $L_{t,c}(\tau)$, are finite dimensional.
\end{prop}
\begin{proof}
%Let $((S\h^*)^G)^p_+ M_{1,c}(\tau) \subset \Ker B$. 
As explained in section \ref{notation}, the Hilbert series of a baby Verma module is defined as
$$\mathrm{Hilb}_{N_{1,c}(\tau)}(z)=\sum_i \dim N_{1,c}(\tau)_iz^i, \,\,\,\,\,\,\, \mathrm{Hilb}_{N_{0,c}(\tau)}(z)=\sum_i \dim N_{0,c}(\tau)_iz^i.$$ The series at $t=0$ and $t=1$ are related by $$\mathrm{Hilb}_{N_{1,c}(\tau)}(z)=\left( \frac{1-z^p}{1-z} \right)^n \mathrm{Hilb}_{N_{0,c}(\tau)}(z^p).$$

Because of this formula, and because $L_{t,c}(\tau)$ is a quotient of $N_{t,c}(\tau)$, it is enough to prove the proposition for $N_{0,c}(\tau)$.

The representation $\tau$ is finite-dimensional, so $M_{0,c}(\tau) \cong S\h^* \otimes \tau$ is a finite module over $S\h^*$.  By Lemma \ref{hilbertnoether}, $S\h^*$ is a finite module over $(S\h^*)^G$. %Since $\h^*$ is finite-dimensional, $(S\h^*)^G$ is a finite module over $((S\h^*)^G)^p$.

For any commutative ring $R$, maximal ideal $\mf{m}$, and finite $R$-module $M$, the quotient $M/\mf{m}M$ is a finite-dimensional vector space over $R/\mf{m}$.  Applying this to $\mf{m} = ((S\h^*)^G)_+$, $R = ((S\h^*)^G)$, and $M = M_{0,c}(\tau)$, it follows that $N_{0,c}(\tau)$ is finite dimensional over $\Bbbk$.
\end{proof}

\begin{rem} \label{sumisall} As in characteristic zero, the module $L_{t,c}(\tau)$ is irreducible. In characteristic zero one shows this by showing that $\Ker B$ is the sum of all graded proper submodules of $M_{t,c}(\tau)$, and that there is a natural inner $\mathbb{Z}$-grading on $M_{t,c}(\tau)$, and that all submodules are graded. In characteristic $p$ this fails, as there is only a natural inner $\mathbb{Z}/p\mathbb{Z}$-grading. There exist submodules of $M_{t,c}(\tau)$ which are not $\mathbb{Z}$-graded: for example, for any $f\in (S\h^*)^G$, the subspace $S\h^*(1+f^p)\otimes \tau$ is a proper submodule. In fact, the sum of all proper submodules of $M_{t,c}(\tau)$ is the whole $M_{t,c}(\tau)$ (the sum of all submodules of the form  $S\h^*(1+f^p)\otimes \tau$ equals $M_{t,c}(\tau)$). However, the situation is better if we consider only graded submodules, or if we let baby Verma modules take over the role of Verma modules. This is explained more precisely by the following results.
\end{rem}

To show irreducibility of $L_{t,c}(\tau)$, we will need the following form of Nakayama's lemma.  Recall that the \emph{Jacobson radical} of a commutative ring $R$, denoted $\rad(R)$, is the maximal ideal that annihilates all simple modules, or equivalently, the intersection of all maximal ideals. Also recall that $\Bbbk[[x_1, \ldots, x_n]]$ is local, so $\rad(\Bbbk[[x_1, \ldots, x_n]]) = \langle x_1, \ldots, x_n \rangle$.

\begin{lemma}[Nakayama]\label{nakayama}
Let $R$ be a commutative ring, $I \subset \rad(R)$ an ideal, and $M$ a finitely generated $R$-module.  Let $m_1, \ldots, m_n \in M$ be such that their projections generate $M/IM$ over $R/I$.  Then, $m_1, \ldots, m_n$ generate $M$ over $R$.
\end{lemma}

\begin{lemma}\label{nakayamacor}
Let $L_{t,c}(\tau)_+$ be the  positively graded part of $L_{t,c}(\tau)$. If $v_1, \ldots, v_m \in L_{t,c}(\tau)$ are such that their projections $\bar{v}_1, \ldots, \bar{v}_m \in L_{t,c}(\tau)/L_{t,c}(\tau)_{+} \cong \tau$ span $\tau$ over $\Bbbk$, then $v_1, \ldots, v_m $ generate $L_{t,c}(\tau)$ as an $S\h^*$ module.
\end{lemma}
\begin{proof}
This is a direct application of Nakayama's lemma, with $R = \Bbbk[[x_1, \ldots, x_n]]$, $M = L_{t,c}(\tau)$, and $I = \langle x_1, \ldots, x_n \rangle = \rad(R)$.  By Lemma \ref{nakayama}, $v_1, \ldots, v_m$ generate $L_{t,c}(\tau)$ as a $\Bbbk[[x_1, \ldots, x_n]]$-module. Since $L_{t,c}(\tau)$ is finite-dimensional, an infinite power series really acts on $M$ as a finite polynomial.
\end{proof}

\begin{prop}\label{lctau-irred}
The $L_{t,c}(\tau)$ are irreducible for every $t, c$ and $\tau$.
\end{prop}
\begin{proof}
Let $f$ be any nonzero element of $L_{t,c}(\tau)$. We claim that it generates $L_{t,c}(\tau)$ as an $H_{t,c}(G,\h)$ module.

If the projection $\bar{f}$ of $f$  to $L_{t,c}(\tau)_{0}\cong \tau$ is nonzero, then the set of $G$-translates of $\bar{f}$ spans the irreducible representation $\tau$, so by Lemma \ref{nakayamacor} the set of $G$-translates of $f$ generates $L_{t,c}(\tau)$ as an $S\h^*$ module, and $f$ generates $L_{t,c}(\tau)$ as an $H_{t,c}(G,\h)$-module.

If the projection of $f$ to $L_{t,c}(\tau)_{0}\cong \tau$ is zero, write $f = f_1 + \cdots + f_d$, with $f_i \in L_{t,c}(\tau)_{i}$ . The form $B$ is non-degenerate on $L_{t,c}(\tau)$, so $f\notin \Ker B$; it respects the grading so there is some $r > 0$ such that $f_r \not\in \Ker B$. The form is bilinear, so there exist a monomial $y_1^{a_1}\cdots y_n^{a_n} \in S^r\h$ and $v\in \tau^*$ such that $B(f_r, y_1^{a_1}\cdots y_n^{a_n}v) \ne 0$. By contravariance of $B$, and writing $D_{i}$ for $D_{y_{i}}$, this is equal to $0\ne B(D_1^{a_1}\cdots D_n^{a_n} f_r,v)= B(D_1^{a_1}\cdots D_n^{a_n} f,v)$. So, $D_1^{a_1}\cdots D_n^{a_n}f$ is a nonzero element of  $L_{t,c}(\tau)$, with a nonzero projection to $L_{t,c}(\tau)_{0}\cong \tau$. By the previous reasoning,  $D_1^{a_1}\cdots D_n^{a_n}f$ generates $L_{t,c}(\tau)$ as an $H_{t,c}(G,\h)$-module, and thus $f$ generates $L_{t,c}(\tau)$ as an $H_{t,c}(G,\h)$-module.
\end{proof}

\begin{cor}
A Verma module $M_{t,c}(\tau)$ has a unique maximal graded submodule.
\end{cor}
\begin{proof}
Consider the sum of all graded submodules.  None of these submodules have elements in $L_{t,c}(\tau)_{0}\cong \tau$, since such elements generate the entire module $M_{t,c}(\tau)$, so their sum is a proper submodule of $M_{t,c}(\tau)$.
\end{proof}
\begin{cor}
A baby Verma module $N_{t,c}(\tau)$ has a unique maximal submodule.
\end{cor}
\begin{proof}
Let $N$ be any proper submodule, and $f\in N$ arbitrary nonzero element. Write $f = f_0 + \cdots + f_d $, with $f_i$ in the $i$-th graded piece. Baby Verma modules are finite-dimensional, $N$ is a proper submodule, so a similar argument as in Proposition \ref{lctau-irred} implies that $f_0 = 0$. Thus, any proper submodule has zero projection to the zeroth graded piece, and so the sum of all proper submodules is still proper.
\end{proof}

The unique maximal graded submodule of $M_{t,c}(\tau)$ descends to the unique maximal submodule of $N_{t,c}(\tau)$, which we will call $\bar{J}_{t,c}(\tau)$.  The following corollary follows by the irreducibility of $M_{t,c}(\tau)/\Ker B $.
\begin{cor}
The kernel of $B$ is $J_{t,c}(\tau)$.
\end{cor}
Thus, $$L_{t,c}(\tau) = M_{t,c}(\tau)/J_{t,c}(\tau)  = M_{t,c}(\tau)/\Ker B \cong N_{t,c}(\tau)/\Ker  B=N_{t,c}(\tau)/\bar{J}_{t,c}(\tau).$$

\subsection{Category $\mathcal{O}$}\label{catO}

We now define the category $\mathcal{O}$ of $H_{t,c}(G, \h)$ modules. The definition, which is somewhat different than in characteristic zero, is justified by Remark \ref{sumisall} and Proposition \ref{all}. 

%As in the characteristic zero case, we will be working in a subcategory of $H_{1,c}$-representations, The motivation for defining $\mathcal{O}$ in characteristic zero is that all irreducibles in this category arise uniquely as some $L_{1,c}(\tau)$ \cite{etingof-ma}, and our definition was formulated analogously.
\begin{defn}
The category $\mathcal{O}_{t,c}(G, \h)$ is the category of $\Z$-graded $H_{t,c}(G, \h)$-modules which are finite-dimensional over $\Bbbk$.
\end{defn}
We usually write $\mathcal{O}_{t,c}$ or $\mathcal{O}$ instead of  $\mathcal{O}_{t,c}(G, \h)$ when it is clear what the arguments are.

\begin{prop}\label{all}
For every irreducible $L \in \mathcal{O}_{t,c}(G, \h)$, there is a unique irreducible $G$-representation $\tau$ and $i \in \mathbb{Z}$ such that $L \cong L_{t,c}(G, \h, \tau)[i]$.
\end{prop}
\begin{proof}
Let $L \in \mathcal{O}_{t,c}$ be any irreducible module in category $\mathcal{O}$. It is graded and finite-dimensional, so there must be a lowest graded piece $L_i$.  Without loss of generality, we can shift indices so that the lowest graded piece is in degree zero. Further, if the degree zero part $L_0$, which is a $G$-representation, is reducible, then any proper $G$-subrepresentation of $L_0$ generates a proper $H_{t,c}(G,\h)$-subrepresentation of $L$. So, $L_0\cong \tau$ for some irreducible $G$-representation $\tau$.  By Proposition \ref{homexist}, there exists a nonzero graded homomorphism $\phi: M_{t,c}(\tau) \rightarrow L$.  Since $L$ is irreducible, this homomorphism is surjective, and $L$ is isomorphic to $M_{t,c}(\tau)/\Ker(\phi)$.  Since $J_{t,c}(\tau)$ is the unique maximal graded submodule, $\Ker(\phi) = J_{t,c}(\tau)$ and the result follows.  Uniqueness follows from the fact that $L_{t,c}(\tau)_i \cong \tau$.
\end{proof}

\subsection{A lemma about finite fields}

\begin{lemma}\label{ntuplesum}
Let $q = p^r$ be a prime power.  Let $f \in \Bbbk[x_1, x_2, \ldots, x_n]$ be a polynomial in $n$ variables, for which there exists a variable $x_i$ such that $$\deg_{x_i}(f) < q-1.$$  %Let $\mathcal{M}_n$ be the set of all $n$-tuples in $\F_q$.  Then, $\sum_{M \in \mathcal{M}_n} f(M) = 0$.
Then $$\sum_{x_1,\ldots x_n \in \F_q}f(x_1,\ldots x_n)=0. $$
\end{lemma}

\begin{proof}
It is enough to prove the claim for monomials of degree $m<q-1$ in one variable.  Let $f=x^m$ and $S_m := \sum_{i \in \F_q} i^m$. For every $j\in \F_q$, $ j^m S_m = \sum_{i \in \F_q} (ij)^m = S_m,$ which is equivalent to $(1 - j^m)S_m = 0$ for all  $j \ne 0$. As $m<q-1$, there exists some $j$ such that $1 - j^m\ne 0$, and so $S_m = 0$.

% and the claim is true for polynomials in one variable. 

%If $n>1$ and $f=f'\cdot  x_n^{m}$, with $f'$ a polynomial in $x_1,\ldots x_{n-1}$ and $m<q-1$, then
%$$\sum_{x_1,\ldots x_n \in \F_q} f(x_1,\ldots x_n) = \sum_{x_1, \ldots, x_{n-1} \in \F_q } f'(x_1,\ldots x_{n-1}) \cdot \sum_{x_n \in \F_q} %x_n^{m} =$$ $$=\sum_{(x_1, \ldots, x_{n-1}) \in \mathcal{M}_{n-1}} f'(x_1,\ldots x_{n-1}) \cdot S_{m} = 0.$$
\end{proof}

\begin{rem}
In particular, the assumptions of the lemma are satisfied by all $f$ such that $\deg(f) < n(q-1).$
\end{rem}

\section{Characters} \label{characters}

\subsection{Definition and basic properties}
\begin{defn}\label{defchar}
Let $K_0(G)$ be the Grothendieck group of the category of finite dimensional representations of $G$ over $\Bbbk$. For $M=\oplus_{i}M_i$ any graded $H_{t,c}(G,\h)$ module with finite dimensional graded pieces, define its character to be the power series in formal variables $z, z^{-1}$ with coefficients in $K_0(G)$
$$\chi_{M}(z)=\sum_{i}[M_{i}] z^i,$$ and recall we defined its Hilbert series as $$\mathrm{Hilb}_{M}(z)=\sum_{i}\dim(M_{i})z^i.$$
\end{defn} If $M$ is in category $\mathcal{O}$, it is finite dimensional and its character is in $K_0(G)[z,z^{-1}]$.

The character of $M_{t,c}(\tau)$ is $$\chi_{M_{t,c}(\tau)}(z)=\sum_{i\ge 0}[S^{i}\h^* \otimes \tau]z^i,$$ and its Hilbert series is $$\mathrm{Hilb}_{M_{t,c}(\tau)}(z)=\frac{\dim(\tau)}{(1-z)^n}.$$ The character of $N_{t,c}(\tau)$ depends on whether $t=0$ or $t\ne0$; these cases are related by $$\chi_{N_{1,c}(\tau)}(z)=\chi_{N_{0,c}(\tau)}(z^p)\cdot \left( \frac{1-z^p}{1-z} \right)^n.$$ 

If $G$ is a reflection group for which the algebra of invariants $(S\h^*)^G$ is a polynomial algebra with homogeneous generators of degrees $d_1,\ldots d_n$, then the characters of baby Verma modules are:
$$\chi_{N_{0,c}(\tau)}(z)=\chi_{M_{0,c}(\tau)}(z)(1-z^{d_1})(1-z^{d_2})\ldots (1-z^{d_n}),$$
$$\chi_{N_{1,c}(\tau)}(z)=\chi_{M_{1,c}(\tau)}(z)(1-z^{pd_1})(1-z^{pd_2})\ldots (1-z^{pd_n}).$$

There is no known general formula for the characters of the irreducible modules $L_{t,c}(\tau)$, even in characteristic zero. The main focus of the second half of this paper is describing these modules for particular series of groups $G$, in terms of their characters, or through describing the generators for the maximal proper submodules $J_{t,c}(\tau)$, or through describing the composition series of baby Verma modules and Verma modules. 

It is clear from the definition that $$\mathrm{Hilb}_{L_{t,c}(\tau)}(z) = \sum_{i=0}^\infty \rank(B_i) z^i.$$ This is useful because matrices $B_{i}$ and their ranks can be calculated in many examples using algebra software. We used MAGMA \cite{magma} to do these calculations for small examples in order to form conjectures which became sections \ref{glnsec} and \ref{slnsec} of this paper. Some unused computational data of this kind can be found in the Appendix \ref{secdata}.

%We now begin a discussion of the Hilbert series, which we will call the \emph{character}, of various $L_{1,c}(\tau)$.  In particular, we can consider $c$ to be a point in $\Bbbk^n$, and we have the notion of a \emph{generic} character of $T \subset \Bbbk^n$, i.e. the character associated with taking $J_{1,c}(\tau)$ to be the elements which are in the kernel of $B$ for all $c \in T$.  Without mention of a particular subset $T \subset \Bbbk^n$, it is assumed that we take $T = k^n$, i.e. the generic character will mean taking $J_{1,c}(\tau)$ to be elements in the kernel of $B$ for all $c$.  

\subsection{Characters of $L_{t,c}(\tau)$ at generic value of the parameter $c$}

By definition, the $i$-th graded piece of $L_{t,c}(\tau)$ is, as a representation of $G$, equal to the quotient of $S^{i}\h^*\otimes \tau$ by the kernel of $B_{i}$. Let us fix $t$ and consider $c=(c_s)_s$ as variables; $B_i$ depends on them polynomially. Let $\Bbbk^{|conj|}$ be the space of functions from the finite set of conjugacy classes in $G$ to $\Bbbk$, and think of it as the space of all possible parameters $c$. 

Let $d$ be the dimension of $S^{i}\h^*\otimes \tau$ and let $r$ be the rank of $B_i$, seen as an operator over $\Bbbk[c]$. For $c$ outside of finitely many hypersurfaces in $\Bbbk^{|conj|}$, the rank of $B_i$ evaluated at $c$ is equal to $r$, and the kernel of $B_i$ is some $(d-r)$-dimensional representation of $G$, depending on $c$. All these representations have the same composition series. (To see this, let $V(c)$ be a flat family of $G$-representations, for example $\Ker B_i$ for generic $c$. Let $\{ \sigma_i \}_i$ be a complete set of pairwise non-isomorphic irreducible $\Bbbk G-$ modules, and for all $i$ let $\pi_i$ be a projective cover of $\sigma_i$. Then the number $[V(c):\sigma_i]$ of times $\sigma_i$ appears as a composition factor in $V(c)$ is equal to the dimension of $\Hom (\pi_i, V(c))$. So, for generic $c$ it is the same, and for special $c$ it might be bigger. But $\sum_i [V(c):\sigma_i]\dim (\sigma_i)=\dim V(c)$ is constant, so $[V(c):\sigma_i]$ does not depend on $c$, and all $V(c)$ have the same factors in their composition multiplicities. They might however not be isomorphic, because they might be different extensions of their irreducible composition factors.)

The map $c\mapsto \Ker (B_i)=J_{t,c}(\tau)_i$, defined on the open complement of hypersurfaces in  $\Bbbk^{|conj|}$, can be thought of as a rational function from $\Bbbk^{|conj|}$ to the Grassmannian of $(d-r)$-dimensional subspaces of $S^{i}\h^*\otimes \tau$. 

For $c$ in some finite family of  hypersurfaces in the parameter space $\Bbbk^{|conj|}$, the rank of $B_i$ evaluated at $c$ is smaller than $r$, and the dimension of the kernel $J_{t,c}(\tau)_i$ is larger than $d-r$. We will now use the above rational function to define a subspace $J_{t,0}(\tau)_i'\subseteq J_{t,0}(\tau)_i$ at $c=0$, which has similar properties to those $J_{t,0}(\tau)_i$ would have if $c=0$ was a generic point.

If $c=0$ is generic and rank of  $B_i$ at $c=0$ is $d$, let $J_{t,0}(\tau)_i'= J_{t,0}(\tau)_i$. Otherwise, pick a line in the parameter space $\Bbbk^{|conj|}$ which does not completely lie in one of the  hypersurfaces, and which passes through $0$. The composition of the inclusion of this line to $\Bbbk^{|conj|}$ and the rational map from $\Bbbk^{|conj|}$ to the Grassmannian is then a rational map from the punctured line to a projective space, and such a map can always be extended to a regular map on the whole line. This associates to $c=0$ a vector space $J_{t,0}(\tau)_i'$. It generally depends on the choice of a line in the parameter space, and it always has the following properties:
\begin{itemize}
\item $\dim(J_{t,0}(\tau)')_i=r-d$;
\item $J_{t,0}(\tau)_i'\subseteq \Ker B_i=J_{t,0}(\tau)_i$;
\item $J_{t,0}(\tau)'_i$ is $G$-invariant;
\item $J_{t,0}(\tau)'_i$ has the same composition series as $J_{t,c}(\tau)_i$ for generic $c$.
\end{itemize}
By making consistent choices for all $i$ (for example, by choosing the same line in the parameter space for all $i$), one can ensure an extra property:
\begin{itemize}
\item $J_{t,0}(\tau)'=\oplus_i J_{t,0}(\tau)'_i$ is a $H_{t,0}(G,\h)$ subrepresentation of $M_{t,0}(\tau)$.
\end{itemize}

This produces a subrepresentation $J'_{t,0}(\tau)$ at $c=0$ such that the quotient $M_{t,0}(\tau)/J_{t,0}(\tau)'$ behaves like $L_{t,c}(\tau)$ at generic $c$, even when $c=0$ is not generic. In particular, $M_{t,0}(\tau)/J_{t,0}(\tau)'$ and $L_{t,c}(\tau)$ at generic $c$ have the same character. 

\begin{example}
For $G=GL_{2}(\F_2)$, $\tau=\triv$, the form $B$ restricted to $M_{1,c}(\mathrm{triv})_4\cong S^{4}\h^*$ has a matrix, written here in the ordered basis $(x_1^4, x_1^3x_2, x_1^2x_2^2, x_1x_2^3, x_2^4)$:
$$B_4 = \left( \begin{array}{ccccc} c^2(c+1)&c^2(c+1)&c^2(c+1)&c^2(c+1)&0\\c^2(c+1)&c(c+1)&0&0&c^2(c+1)\\c^2(c+1)&0&0&0&c^2(c+1)\\c^2(c+1)&0&0&c(c+1)&c^2(c+1)\\0&c^2(c+1)&c^2(c
+1)&c^2(c+1)&c^2(c+1) \end{array}\right).$$
When $c\ne 0,1$, this matrix has rank $4$, and a one-dimensional kernel $J_{1,c}(\triv)_4$ spanned by $x_1^4+x_1^2x_2^2+x_2^4$. For $c=0$, the matrix is zero and $J_{1,0}(\triv)_4$ is the whole $S^4\h^*$. The above procedure defines $J_{1,0}(\triv)'_4$ to be $\Bbbk(x_1^4+x_1^2x_2^2+x_2^4)$. 
\end{example}

We will now draw conclusions about the character of $L_{t,c}(\tau)$ for generic $c$ using information about $M_{t,0}(\tau)/J_{t,0}'(\tau)$.

\begin{lemma}\label{pthpowergen}
 Let $M$ be a free finitely-generated graded $S\h^*$-module with free generators $b_1, \ldots, b_m$, and $N$ a graded submodule of $M$. For $f\in S\h^*$, $y\in \h$, define $\partial_y f b_i = (\partial_y f) b_i $. If $N$ is stable under $\partial_y$ for all $y\in \h$, then it is generated by  elements of the form $\sum f_i^p b_i$ for some $f_i \in S\h^*$.
\end{lemma}

\begin{proof}
First, assume there is only one generator, so $M \cong S\h^*$ as left $S\h^*$ modules.  Let $N' = \{f^p \mid f \in S\h^*\} \cap N$.  We claim that $S\h^*N' = N$. 

Clearly, $S\h^*N' \subset N$. To show that $N \subset S\h^* N'$, we need to show that any $f\in N$ can be written as a sum of elements of the form $h(x_1, \dots, x_n) f'(x_1^p, \ldots, x_n^p)$, for some  $h\in S\h^*$ and $f'(x_1^p, \ldots, x_n^p) \in N$.

As $N$ is graded, assume $f$ is homogeneous of degree $d$. Write it as
$$f = \sum_{i=0}^{p-1} x_1^i f_i(x_1^p, x_2, \ldots, x_n).$$ The space $N$ is stable under all partial  derivatives, so for each $j=0,\ldots, p-1$,
$$x_1^j \partial_1^j f = \sum_{i=1}^{p-1} i(i-1)\ldots(i-j+1) x_1^i f_i(x_1^p, x_2, \ldots, x_n) \in N.$$
The coefficient $i(i-1)\ldots(i-j+1)$ is zero for $i < j$ and is nonzero for $i = j$, so the matrix $[i(i-1)\ldots(i-j+1)]_{i,j}$ is  invertible, implying that $x_1^i f_i(x_1^p, x_2, \ldots, x_n)$ is in $N$ for all $i$, and therefore (after applying $\partial_{1}^i$), also $f_i(x_1^p, x_2, \ldots, x_n) \in N$.
 
Applying the same argument on each $f_i$ for $x_2, \ldots, x_n$, it follows that $f$ is of desired form. The claim for $M\cong \oplus S\h^* b_{i}$ follows directly from the one for $S\h^*$.
\end{proof}

Let $S^{(p)}\h^*$ be the quotient of $S\h^*$ by the ideal generated by $x_1^p,\ldots x_n^p$.

\begin{prop}\label{reducedchar}
The character of $L_{1,c}(\tau)$, for generic  value of $c$, is of the form $$\chi_{L_{1,c}(\tau)}(z) =\chi_{S^{(p)}\h^*}(z) H(z^p)$$ for $H\in K_{0}[z]$ the character of some graded $G$-representation. In particular, the Hilbert series of $L_{1,c}(\tau)$ is of the form 
$$\mathrm{Hilb}_{L_{1,c}(\tau)}(z) =\left( \frac{1-z^p}{1-z}\right)^n\cdot h(z^p),$$ for $h$ a polynomial with nonnegative integer coefficients. 
\end{prop}
\begin{proof}
As commented above, the character of $L_{1,c}(\tau)=M_{1,c}(\tau)/J_{1,c}(\tau)$ is the same for all $c$ outside of finitely many  hypersurfaces, and it is equal to the character of $M_{1,0}(\tau)/J_{1,0}(\tau)'$. At these values of parameter, $t=1$ and $c=0$, Dunkl operators are particularly simple, and equal to partial derivatives: $D_{y}=\partial_y$. By the previous lemma, $J_{1,0}(\tau)'$ is generated by $p$-th powers. Let $f_{i}(x_1^p,\ldots x_n^p)\otimes v_i$, for some $f_i\in S\h^*$, $v_i\in \tau$, be these generators. 

Define $J^*$ to be the $(S\h^*)^p$-module generated by $f_{i}(x_1^p,\ldots x_n^p)\otimes v_i$. Let the \emph{reduced module}  $R_{t,c}(\tau)$ be the $\Bbbk [G] \ltimes S\h^*$-module defined as the quotient of $S\h^*\otimes \tau$ by the ideal generated by $f_{i}(x_1,\ldots x_n)\otimes v_i$. Call its character (in the sense of Definition \ref{defchar}) the \emph{reduced character} of $L_{t,c}(\tau)$, and denote it by $H(z)\in K_0(G)[z]$.

Consider the multiplication map
$$\mu: S^{(p)}\h^* \otimes ((S\h^*)^p \otimes \tau)/J^* \to S\h^*\otimes \tau/J_{1,0}(\tau)'.$$ It is an isomorphism of graded $G$-representations, so it preserves characters. From this it follows that for generic  $c$,
 $$\chi_{L_{1,c}(\tau)}(z)=\chi_{M_{1,0}(\tau)/J_{1,0}(\tau)'}(z)=\chi_{S^{(p)}\h^*}(z) H(z^p).$$

\end{proof}

By inspecting the proof and using that $c$ is non-generic on a union of finitely many  hypersurfaces,  one can strengthen the claim of the proposition as follows: for any hyperplane $P$ passing through the origin in the space of functions from the conjugacy classes of $G$ to $\Bbbk$, there exists a function $H_P(z)\in K_0(G)[z]$ such that, for $c$ generic in $P$, the character of $L_{t,c}(\tau)$ is of the form $\chi_{S^{(p)}\h^*}(z) H_P(z^p)$.

\subsection{A dimension estimate for $L_{1,c}(\tau)$}

\begin{lemma}
Any irreducible $H_{1,c}(G,\h)$-representation has dimension less than or equal to $p^n |G|$.
\end{lemma}
\begin{proof}

We begin with a definition, which will only be used in this proof.  Let $A$ be an algebra.  A \emph{polynomial identity} is a nonzero  polynomial $f(x_1, \ldots, x_r)$ in non-commuting variables $x_1,\ldots x_r$, with a property that $f(a_1, \ldots, a_r) = 0$ for all $a_1, \ldots, a_r \in A$.  Given an algebraically closed field $\Bbb\Bbbk$, a \emph{polynomial identity algebra}, or \emph{PI algebra} is a $\Bbbk$-algebra $A$ that satisfies a polynomial identity.  We say a PI algebra has degree $r$ if it satisfies the polynomial identity $s_{2r} = \sum_{\sigma \in S_{2r}}\sgn(\sigma) \prod_{i=1}^{2r} x_{\sigma(i)}$.

Our first claim is that $H_{1,c}(G,\h)$ is a PI algebra.  By Proposition V.5.4 in \cite{artin}, $A$ is a PI algebra if and only if every localization of $A$ is also a PI algebra.  By the localization lemma (Lemma \ref{loclemma}), $H^\loc_{1,c}(G,\h) = H^\loc_{1,0}(G,\h)$.  Thus, it suffices to show that $H_{1,0}(G,\h)$ is a PI algebra.

Let $Z$ be the center of $H_{1,0}(G,\h)$.  It is easy to see that $Z = (S(\h\oplus \h^*)^p)^G$.  $Z$ is commutative, so we can consider $A' = \Frac(Z) \otimes_Z H_{1,0}(G,\h)$, which is an algebra over the field $\Frac(Z)$.  By Theorem V.8.1 in \cite{artin}, this is a \emph{central simple algebra}, i.e. an algebra that is finite-dimensional, simple, and whose center is exactly its field of coefficients.  By the Artin-Wedderburn theorem, a central simple algebra is isomorphic to a matrix algebra over a division ring.  To calculate the size of these matrices, we would like to determine the dimension of $A'$.
We need to calculate
$$\dim_{\Frac(Z)} A'  = \dim_{\Frac((S(\h\oplus \h^*)^p)^G)}\Frac((S(\h\oplus \h^*)^p)^G) \otimes_{(S(\h\oplus \h^*)^p)^G}H_{1,0}(G,\h) $$
Using the PBW theorem and the fact that for all algebras $B$ and $C$, $B\cong C\otimes _C B$, we get that this is equal to
$$ \dim_{\Frac((S(\h\oplus \h^*)^p)^G)} \Frac((S(\h\oplus \h^*)^p)^G) \otimes_{(S(\h\oplus \h^*)^p)^G} S(\h\oplus \h^*)^p 
 \otimes_{S(\h\oplus \h^*)^p} S(\h\oplus \h^*)\otimes \Bbbk G.$$

Next, we use the following fact: if $G$ acts on some vector space $V$ (in our case, $V=(\h\oplus \h^*)^p$) in such a way that there exists a vector $v\in V$ with a trivial stabilizer, then the $G-$orbit of $v$ is a $\Frac(SV^G)-$ basis for $\Frac(SV^G) \otimes_{SV^G} SV$, and $\dim_{\Frac(SV^G)}\Frac(SV^G) \otimes_{SV^G} SV=|G|$. 

We get that the dimension above is equal to 
$$ \dim_{\Frac((S(\h\oplus \h^*)^p)^G)} \Frac((S(\h\oplus \h^*)^p)^G) \otimes_{(S(\h\oplus \h^*)^p)^G} S(\h\oplus \h^*)^p 
 \otimes_{S(\h\oplus \h^*)^p} S(\h\oplus \h^*)\otimes \Bbbk G=$$
$$=|G|\cdot p^{\dim{S\h\oplus\h^*}} \cdot |G|=|G|^2p^{2n}. $$

So, $A'$ is isomorphic to the matrix algebra over $\Frac((S(\h\oplus \h^*)^p)^G)$ of size $|G|p^{n}$. By Corollary V.8.4 in \cite{artin}, an $r \times r$ matrix algebra satisfies $s_{2r}$, so $A'$ is a PI algebra of degree $p^n|G|$.  Consequently,  $H_{1,0}(G,\h)$ is a PI algebra of degree $p^n|G|$, and so is  $H_{1,c}(G,\h)$.  By Proposition V.6.1(ii) in \cite{artin}, any irreducible representation of a PI algebra of degree $r$ has dimension less than or equal to $r$, and the result follows.
\end{proof}

\begin{cor}\label{lcdim}
Let $h$ be the reduced Hilbert series of $L_{1,c}(\tau)$ for generic $c$. Then $L_{1,c}(\tau)$ has dimension $h(1) p^n$, and $1 \leq h(1) \leq |G|$.
\end{cor}

\subsection{Some observations, questions and remarks}

%\begin{conj}\label{conj00}
%Let $h$ be the reduced Hilbert series of $L_{1,c}(\tau)$ for generic $c$. Then  $h(1) \mid |G|$.
%\end{conj}
%In many cases we consider, $h(1) = 1$ or $h(1) = |G|$.  For $G = S_4$ over $\F_3$, we expect $h(t) = \frac{(1 - t^2)^2(1 - t^3)}{(1 - t)^3}$, so $h(1) = 12 \ne |S_4|$.

\begin{rem}\label{conj00}
In many examples we considered, in particular whenever $G=GL_n(\F_q)$ or $G=SL_n(\F_q)$ and $\tau=\mathrm{triv}$, $h(1)$ is equal to $1$ or to $|G|$. In many other cases, it divides $|G|$. However, this is not always true. For $G=GL_2(\F_p)$, $\tau=S^{p-2}\h$, the order of the group is $(p^2-1)(p^2-p)$, and the reduced Hilbert series is $p+(p-1)z+pz^2$. So, $h(1)=3p-2$, which does not always divide  $(p^2-1)(p^2-p)$ (for example, when $p=3$).
\end{rem}

\begin{quest}\label{conj01}
For $h_1(z)=\sum_i a^{(1)}_iz^i$  the reduced Hilbert series of $L_{1,c}(\tau)$ and $h_0(z)=\sum_i a^{(0)}_iz^i$ the Hilbert series of $L_{0,c}(\tau)$, does the inequality  $$a^{(0)}_i \le a^{(1)}_i$$ hold?
 \end{quest}
There is computational data supporting the positive answer. In many examples, particularly for $G=GL_{n}(F_q)$ and $SL_{n}(\F_p)$, the equality $h_{0}= h_1$ holds. An example when strict inequality is achieved is $G = SL_2(\F_3)$, $\tau = \triv$: the reduced Hilbert series is $h_1(z) = (1 + z + z^2 + z^3)(1 + z + z^2 + z^3 + z^{4} + z^{5})$, and the Hilbert series of $L_{0,c}(\tau)$ is $h_0(z)=1$.

Recall that a finite dimensional $\mathbb{Z}_{+}$ graded algebra $A=\oplus_i A_{i}$ is  \emph{Frobenius} if the top degree $A_d$ is one dimensional, and multiplication $A_{i}\otimes A_{d-i}\to A_{d}$  is a non-degenerate pairing. As a consequence, the Hilbert series of $A$ is a palindromic polynomial.

The irreducible module $L_{t,c}(\mathrm{triv})$ is a quotient of $M_{t,c}(\mathrm{triv})\cong S\h^*$ by the $H_{t,c}(G,\h)$ submodule $J_{t,c}(\mathrm{triv})$, which is in particular an $S\h^*$ submodule. So, we can consider it as a quotient of the algebra $S\h^*$ by the left ideal $J_{t,c}(\mathrm{triv})$, and therefore as a finite dimensional graded  commutative algebra.

\begin{prop}\label{frobprop}
Assume that $t,c,\tau$ are such that the top graded piece of $L_{t,c}(\mathrm{triv})$ is one dimensional. Then $L_{t,c}(\mathrm{triv})$ is Frobenius. 
\end{prop}
\begin{proof}
Let us first prove: a finite dimensional graded commutative algebra $A=\oplus_{i=0}^d A_{i}$ is Frobenius if and only if the kernel on $A$ of the multiplication by $A_+=\oplus_{i>0} A_{i}$ is one dimensional. One implication is clear: if $A$ is Frobenius, the kernel is the one dimensional space $A_d$. For the other, assume the kernel on $A$ of the multiplication by $A_+$ is one dimensional.  The top nontrivial graded piece $A_d$ is always contained in it, so $A_d$ is one dimensional and equal to the kernel. Now assume there exists a nonzero element $a_n\in A_n$ such that multiplication by $a_n$, seen as a map $A_{d-n}\to A_d$, is zero, and let $0<n<d$ be the maximal index for which such an $a_n$ exists. As $a_n$ is not in the kernel of the multiplication by $A_+$, there exists some $b\in A_+$ such that $a_nb\ne 0$. We can assume without loss of generality that $b$ is homogeneous, $b\in A_m$, $0<m<d-n$. Then $a_nb\in A_{n+m}$, with $n<n+m<d$, is a nonzero element such that multiplication by it, seen as a map $A_{d-n-m}\to A_d$, is zero, contrary to the choice of $n$ as the largest such index.

Now assume that $A=L_{t,c}(\mathrm{triv})$ has one dimensional top degree. Let $0\ne f$ be in the kernel of multiplication by $A_+$. As the kernel is graded, assume without loss of generality that $f$ is homogeneous. Then $xf=0\in L_{t,c}(\mathrm{triv})$ for all $x\in \h^*$, so $x$ is a \emph{highest weight vector}. Under the action of $H_{t,c}(G,\h)$, $f$ generates a subrepresentation of $L_{t,c}(\mathrm{triv})$ for which the highest graded piece consists of $G$-translates of $x$. As $L_{t,c}(\mathrm{triv})$ is irreducible, this subrepresentation has to be the entire $L_{t,c}(\mathrm{triv})$, and $f$ is in the top degree, which is by assumption one dimensional.
\end{proof}

\begin{rem}\label{conj02}
In many instances we observed, the algebra $L_{t,c}(\triv)$ is Frobenius for generic $c$ and has palindromic Hilbert series. However, this is not true in general:  let $\Bbbk=\overline{\F}_3$, $G = S_5$ the symmetric group on five letters, $\h$ the four dimensional reflection representation $\{(z_1,\ldots z_5)\in \Bbbk^5 | z_1+\ldots +z_5=0 \}$ with the permutation action, and $\tau = \triv$. Then  the character of $L_{0,c}(\triv)$ is $$(1 + t)(1 + t + t^2)(1 + 2t + 3t^2 + 4t^3).$$ We thank Sheela Devadas and Steven Sam for pointing out this counterexample to us. 
\end{rem}

%%%%%%THE FOLLOWING IS NEVER USED -- MAYBE STICK IT SOMEWHERE ELSE?

%\begin{prop}\label{dunklkilled}
%Suppose that the smallest $G$-subprepresentation of $S^{p-1}(\mf{h}^*)$ has dimension greater than $n^2$.  Then for any $x \in \mf{h}^*$ and $y \in \mf{h}$, $D_y(x^p) = 0$
%\end{prop}
%\begin{proof}
%Let $E$ be the $G$-representation spanned by $x^p$ for all $x \in \mf{h}^*$.  It is a representation, since $g \ldot x^p = (g \ldot x)^p$ in characteristic $p$.  Since $\{x_1^p, \ldots, x_n^p\}$ is a basis, $E$ has dimension $n = \dim_\Bbbk(\h)$.  Consider $\phi: \h \otimes E \rightarrow S^{p-1}(\h^*)$ where $\phi(y \otimes x^p) = D_y(x^p)$; it is a homomorphism of representations since $D_{g \ldot y}(g \ldot x) = g \ldot D_y(x)$.  Further, $\h \otimes E$ has dimension $n^2$, and $\phi$ must map each irreducible subrepresentation by zero or isomorphism to an irreducible subrepresentation of $S^{p-1}(\h^*)$.  Since no irreducible subrepresentation of $S^{p-1}(\h^*)$ has dimension less than or equal to $n^2$, it must map by zero, proving the claim.
%\end{proof}

\subsection{Invariants and characters of baby Verma modules for $G=GL_n(\F_q)$ and $G=SL_n(\F_q)$}\label{babychar}

Remember that $N_{t,c}(\tau)$ is a quotient of $M_{t,c}(\tau)\cong S\h^*\otimes \tau$ by $(S\h^*)_{+}^G\otimes \tau$ if $t=0$ and by  $((S\h^*)^G)_+^p \otimes \tau$ if $t\ne0$. For all groups $G$ for which we know the Hilbert series of the space of invariants $(S\h^*)^G$, we can calculate the character of baby Verma modules easily. An especially nice case is when $(S\h^*)^G$ is a polynomial algebra generated by algebraically independent elements of homogeneous degrees $d_1,\ldots d_n$. In that case, such elements  are called \emph{fundamental invariants}, and $d_i$ are called \emph{degrees} of $G$. 

In \cite{dickson}, Dickson shows that $GL_{n}(\F_q)$ is such a group; the result for $SL_n(\F_q)$ follows easily and is explained in \cite{kemper-malle}. Let us recall the construction of invariants and the calculation of their degrees before calculating the characters of baby Verma modules. 

Remember that for $G=GL_{n}(\F_q)$ and $G=SL_{n}(\F_q)$, the reflection representation $\h$ is the vector representation $\Bbbk^n$, and so $\h_\F=\F_q^n$, $\h_\F^*=\F_q^n$, $\h^*=\Bbbk^n$. Let $x_1,\ldots,x_n$ be the tautological (coordinate) basis for $\h_{\F}^*$. For an ordered $n$-tuple of non-negative integers $e_1, \ldots, e_n$, let $[e_1, \ldots, e_n] \in S\h^*$ be the determinant of the matrix whose entry in the $i$-th row and $j-th$ column is $x_j^{q^{e_i}}$. The action of $G$ on $\h^*$ is dual to the tautological action, so the matrix of $g\in GL_{n}(\F_q)$ in the basis $(x_i)_i$ is $(g^{-1})^{t}$. Taking determinants is a multiplicative map, so a direct calculation gives that for $g\in GL_n(\F_q)$, $$g \ldot [e_1, \ldots, e_n] = (\det(g))^{-1} [e_1, \ldots, e_n].$$ 

Define $$L_n:= [n-1, n-2, \ldots, 1, 0],$$
$$ Q_{i} := [n, n-1, \ldots, i+1, i-1, \ldots, 1, 0] / L_n, \,\,\,\,i=1,\ldots n-1,$$
and $$Q_0=L_n^{q-1}.$$

The paper \cite{dickson} shows that $[n, n-1, \ldots, i+1, i-1, \ldots, 1, 0]$ is divisible by $L_n$, and so $Q_i$ are indeed in $S\h^*$. From the observation that all $[e_1, \ldots, e_n] $ transform as $(\det g)^{-1}$ under the action of $g \in GL_n(\F_q)$, it follows that $Q_i, i=0,\ldots, n-1$ are invariants in $S\h^*$. The main theorem in \cite{dickson} states a stronger claim:

\begin{thm}\label{gln-inv}
Polynomials $Q_{i}$, $i =0, \ldots, n-1$,  form a fundamental system of invariants for $GL_n(\F_q)$ in $S\h^*$, i.e. they are algebraically independent and generate the subalgebra of invariants.
\end{thm}

A comment in Section 3 of \cite{kemper-malle} gives us the following corollary.
\begin{cor}
Polynomials $L_n$ and $Q_{i}$ for $i = 1, \ldots, n-1$ form a fundamental system of invariants for $SL_n(\F_q)$.
\end{cor}

The degrees of these invariants are:
$$\deg L_n=1+q+\ldots +q^{n-1}$$
$$\deg Q_0=(q-1)\deg L_n= (q-1)(1+q+\ldots +q^{n-1})=q^n-1$$
$$\deg Q_i=(1+q+\ldots +q^n -q^i )-(1+q+\ldots +q^{n-1})=q^n-q^i.$$

From this, we can calculate the characters of the baby Verma modules for these groups. 

\begin{cor}
For $G = GL_n(\F_q)$, the characters of the baby Verma modules are
$$\chi_{N_{0,c}(\tau)}(z)=\chi_{M_{0,c}(\tau)}(z) \prod_{n=0}^{n-1} (1 - z^{q^n - q^i}),$$
$$\chi_{N_{t,c}(\tau)}(z)=\chi_{M_{t,c}(\tau)}(z) \prod_{n=0}^{n-1} (1 - z^{p(q^n - q^i)}), \,\,\,\, t\ne 0.$$

For $G = SL_n(\F_q)$, the characters of the baby Verma modules are
$$\chi_{N_{0,c}(\tau)}(z)=\chi_{M_{0,c}(\tau)}(z)(1-z^{1+q+\ldots +q^{n-1}}) \prod_{n=1}^{n-1} (1 - z^{q^n - q^i}),$$
$$\chi_{N_{t,c}(\tau)}(z)=\chi_{M_{t,c}(\tau)}(z) (1-z^{p(1+q+\ldots +q^{n-1})}) \prod_{n=1}^{n-1} (1 - z^{p(q^n - q^i)}), \,\,\,\, t\ne 0.$$

\end{cor}

\section{Description of $L_{t,c}(\triv)$ for $GL_n(\F_{p^r})$}\label{glnsec}

\subsection{$GL_n(\F_{p^r})$ as a reflection group}

For this section, let us fix $G = GL_n(\F_q)$, for $q = p^r$ a prime power and $n > 1$. It acts tautologically on $\h_{\F}=\F_q^n$, which has a coordinate basis $y_1,\ldots ,y_n$, and by the dual representation on $\h_{\F}^*$, which has a dual basis $x_1,\ldots ,x_n$. The underlying field is $\Bbbk = \overline{\F}_q$, and the reflection representation is $\h=\h_{\F}\otimes_{\F_q}\Bbbk=\Bbbk^n$. We will study $L_{t,c}(\tau)$ for all $t,c$ and the trivial representation $\tau=\Bbbk$, where every element of $G$ acts as identity.

The group $GL_n(\F_q)$ is indeed a reflection group, generated by all the conjugates of elements $d_{\lambda}$ for $\lambda \in \F_{q}^{\times}$ (for definition of $d_{\lambda}$, see Section \ref{notation}). The order of $GL_n(\F_q)$ is $\prod_{i=0}^{n-1} (q^n - q^i)$, which is divisible by the characteristic $p$ of $\Bbbk$. It is an example of a reflection group in characteristic $p$ which has no counterpart in characteristic zero.

\begin{lemma}\label{conjclassGLn}
Reflections in $GL_n(\F_q)$ are elements that are conjugate in $GL_n(\F_q)$ to one of the  $d_\lambda$, $\lambda \in \F_{q}^{\times}$. There are $q - 1$ conjugacy classes of reflections in $GL_n(\F_q)$, with representatives  $d_\lambda$. Each semisimple conjugacy class consists of $\frac{(q^n - 1)q^{n-1}}{q-1}$ reflections.  The unipotent conjugacy class (elements conjugate to $d_1$) consists of $\frac{(q^n - 1)(q^{n-1} - 1)}{(q - 1)}$ reflections.
\end{lemma}
\begin{proof}

Every semisimple reflection in $GL_n(\F_q)$ is diagonalizable, with eigenvalue $1$ of multiplicity $n-1$ and eigenvalue $\lambda^{-1}$ of multiplicity $1$,  for some $\lambda\ne 0,\ne 1$. Such a reflection is conjugate in $GL_n(\F_q)$ to $d_\lambda$. For $\lambda \ne \mu$, $d_{\lambda}$ is not conjugate to $d_{\mu}$. Thus, to every $\lambda \in \F_q, \lambda \ne 0,\ne 1$, we associate a conjugacy class $C_{\lambda}$ containing it. The centralizer of $d_{\lambda}$ is $GL_1(\F_q) \times GL_{n-1}(\F_q)\subseteq GL_n(\F_q)$, so the number of reflections in the conjugacy class $C_{\lambda}$ is $|GL_n(\F_q)|/|GL_{1}(\F_q) \times GL_{n-1}(\F_q)| = \frac{(q^n - 1)q^{n-1}}{q-1}$.

The unipotent conjugacy class $C_1$ is the orbit of $d_{1}$, which is centralized by any element of the form:
\begin{equation*}
\left[ \begin{array}{cccc}
a & b & \cdots & c \\
0 & a & \cdots & 0 \\
\vdots & \vdots & * & * \\
0 & f & * & *
\end{array} \right]
\end{equation*}
Here, $a \not= 0$, the rest of the variables are arbitrary, and the bottom right $(n-1)\times (n-2)$ submatrix is invertible.  The  order of the centralizer is $q^{2n-3}(q-1)|GL_{n-2}(\F_p)|$, and the order of the conjugacy class is $\frac{(q^n - 1)(q^{n-1} - 1)}{(q - 1)}$.
\end{proof}

Let $C_\lambda$ be the conjugacy class containing $d_{\lambda}$, let $c_{\lambda}$ be the value of $c$ on $C_{\lambda}$. Reflection $s\in C_{\lambda}$ has (generalized, if $\lambda=1$) eigenvalues $\lambda,1,\ldots, 1$ on $\h^*$ and  $\lambda^{-1},1,\ldots, 1$ on $\h$.

\begin{example}\label{n2refl}If $n=2$, the parametrization of $C_{\lambda}$ by $\alpha\otimes \alpha^{\vee}\in \h^*\otimes \h$ described in Lemma \ref{reflequiv} is as follows:
$$\lambda\ne 1:\,\,\,\, C_{\lambda}\leftrightarrow \{\left[\begin{array}{c} 1 \\ b \end{array} \right] \otimes \left[\begin{array}{c} 1-\lambda -bd \\ d \end{array} \right]  | b,d\in \F_q  \} \cup \{ \left[\begin{array}{c} 0 \\ 1 \end{array} \right] \otimes \left[\begin{array}{c} a \\ 1-\lambda \end{array} \right]  | a \in \F_q \} $$
$$\lambda= 1:\,\,\,\, C_{1} \leftrightarrow\{\left[\begin{array}{c} 1 \\ b \end{array} \right] \otimes \left[\begin{array}{c} -bd \\ d \end{array} \right]  | b,d\in \F_q, d\ne 0  \} \cup \{ \left[\begin{array}{c} 0 \\ 1 \end{array} \right] \otimes \left[\begin{array}{c} a \\ 0 \end{array} \right]  | a \in \F_q, a\ne 0 \}.$$
\end{example}

%We now begin studying $L_{t,c}(\mathrm{triv})$.

\subsection{Irreducible modules with trivial lowest weight for $GL_{n}(\F_q)$ at $t=0$}

\begin{thm}\label{glnt0}The characters of the irreducible modules $L_{0,c}(\mathrm{triv})$ for the rational Cherednik algebra $H_{0,c}(GL_n(\F_q),\h)$ are:\\

\begin{tabular}{|c|c|c|c|}
\hline
$(q,n)$ & $c$ & $\chi_{L_{0,c}(\mathrm{triv})}(z)$ & $\mathrm{Hilb}_{L_{0,c}(\mathrm{triv})}(z)$ \\ \hline \hline
$(q,n)\ne (2,2)$ & any & $[\triv]$ & $1$ \\ \hline
$(2,2)$ & $0$ & $[\triv]$ & $1$ \\ \hline
$(2,2)$ & $c\ne0$ & $[\triv]+[\h^*]z+([S^2\h^*]-[\triv])z^2+$ & $1+2z+2z^2+z^3$\\ 
& & $+([S^3\h^*]-[\h^*]-[\triv])z^3$ & \\ \hline
\end{tabular}
\end{thm}

\begin{proof}
We claim that when $(n,q)\ne(2,2)$, all the vectors $x\in \h^*\otimes \mathrm{triv}\cong M_{0,c}(\mathrm{triv})_1$ are singular. To see that, remember that the Dunkl operator associated to $y\in \h$ is 
$$D_{y}= t\partial_y \otimes 1 - \sum_{s \in S} c_s \frac{(y, \alpha_s)}{\alpha_s} (1 - s) \otimes s,$$ which for $t=0$ and $\tau=\mathrm{triv}$ becomes the operator
$$D_{y}=- \sum_{s \in S} c_s \frac{(y, \alpha_s)}{\alpha_s} (1 - s)$$ on $M_{0,c}(\mathrm{triv})\cong S\h^*$. We claim that for all values of the parameter $c$, for all $x\in \h^*$ and for all $y\in h$, the value $D_y(x)=0$.  To see that, calculate the coefficient of $c_{\lambda}$ in $D_y(x)$. Using the parametrization of conjugacy classes from Proposition \ref{reflequiv}, this coefficient is equal to
$$-\sum_{s\in C_{\lambda}}\frac{(y, \alpha_s)}{\alpha_s} (1 - s)\ldot x= -\sum_{\substack{\alpha\otimes \alpha^{\vee}\ne 0 \\ (\alpha,\alpha^{\vee})=1-\lambda}} (y, \alpha)(\alpha^{\vee}, x)=-\sum_{\alpha}(\alpha,y)\left( x,\sum_{(\alpha,\alpha^{\vee})=1-\lambda}\alpha^{\vee} \right).$$ We claim that for fixed $\alpha \in \h^*$, the sum $\sum_{(\alpha,\alpha^{\vee})=1-\lambda}\alpha^{\vee}$ is zero. Fix $\alpha$ and let us change the coordinates so that $\alpha$ is the first element of some new ordered basis. Write the sum $\sum_{(\alpha,\alpha^{\vee})=1-\lambda}\alpha^{\vee}$ in the dual of this basis. The set of all nonzero $\alpha^{\vee}$ such that $(\alpha,\alpha^{\vee})=1-\lambda$ written in these new coordinates then consists of all $\alpha^{\vee}=(1-\lambda,a_2,\ldots, a_n)\ne 0$, for $a_i\in \F_q$. If $\lambda\ne 1$, the sum of all such $\alpha^{\vee}$ is the sum over all $a_2,\ldots , a_n\in \F_q$, so the sum is zero. If  $\lambda=1$, the first coordinate is always zero, so the sum is over all $a_i\in \F_q$ which are not all simultaneously $0$. However, adding $(0,\ldots ,0)$ doesn't change the sum, which is then equal to $$\sum_{a_2,\ldots , a_n\in \F_q}(0,a_2,\ldots a_n)=(0, q^{n-2}\sum_{a\in \F_q}a ,\ldots, q^{n-2}\sum_{a\in \F_q}a ).$$ This is equal to $0$, as claimed, unless $n=2$ and $q=2$.

In case $(n,q)=(2,2)$, direct computation of matrices of the bilinear form $B$ implies the claim. More precisely, if $c=0$ then all $D_y=0$, the form $B$ is zero in degree one, and the module $L_{0,0}(\mathrm{triv})$ is one dimensional. If $c\ne 0$, then the only vectors in the kernels of matrices $B_i$ are the invariants in degrees $2$ and $3$ and all their $S\h^*$ multiples. This implies that $L_{0,c}(\mathrm{triv})=N_{0,c}(\mathrm{triv})$, and gives the character formula from the statement of the theorem.

\end{proof}

\subsection{Irreducible modules with trivial lowest weight for $GL_{n}(\F_q)$ at $t\ne0$}

For the rest of this section we assume that $t=1$. As the parameters $t$ and $c$ can be simultaneously rescaled, the results we obtain for $t=1$ hold (after rescaling $c$ by $1/t$)  for all $t\ne 0$.

\begin{prop}\label{gln-kernel}
Suppose $(q, n) \not= (2, 2)$.  For any $x \in \h^*$, the vector $x^p$ is singular in $M_{1,c}(\mathrm{triv})$.
\end{prop}
\begin{proof}
Method of proof is an explicit calculation analogous to the proof of Theorem \ref{glnt0}. By definition, 
\begin{eqnarray*}
D_y(x^p) & = & \partial_y(x^p) - \sum_{s \in S} c_s (\alpha_s, y) \frac{1}{\alpha_s} (1 - s) \ldot x^p\\
&=& 0-\sum_{s \in S} c_s (\alpha_s, y) \frac{1}{\alpha_s} ((1 - s) \ldot x)^p\\
\end{eqnarray*}
We will show that for every conjugacy class $C_{\lambda}$ of reflections, the sum $$\sum_{s \in C_{\lambda}} (\alpha_s, y) \frac{1}{\alpha_s} ((1 - s) \ldot x)^p $$ vanishes. By Proposition \ref{reflequiv}, this is equal to

$$\sum_{\substack{\alpha\otimes \alpha^{\vee}\ne 0 \\ (\alpha,\alpha^{\vee})=1-\lambda}} (\alpha, y) \frac{1}{\alpha} ((\alpha^\vee, x) \alpha)^p =\sum_{\alpha\ne 0}(\alpha, y) \alpha^{p-1}   \sum_{\substack{\alpha^{\vee}\ne 0 \\ (\alpha,\alpha^{\vee})=1-\lambda}} (\alpha^\vee, x)^{p} $$

It is enough to fix $\alpha$ and show that the inner sum over $\alpha^\vee$ is zero. After fixing $\alpha$, let us change the basis so that $\alpha$ is the first element of the new ordered basis. The set of all  $\alpha^\vee\ne 0$ such that $ (\alpha,\alpha^{\vee})=1-\lambda$, written in the dual of this new basis, is $A := \{ ((1-\lambda),a_2,\ldots,a_n)\ne 0 \mid a_2, \ldots, a_n \in \F_{q}\}$. For a fixed $x$, the expression
$$\sum_{\alpha^\vee \in A} (x, \alpha^\vee_s)^p $$ is a sum over all possible values of $n-1$ variables $a_{2},\ldots ,a_{n}$ of a polynomial of degree $p$. By Lemma \ref{ntuplesum}, this is zero if $p<(n-1)(q-1)$. This is only violated when $n=2$, $p=q$. In that case $(\alpha^\vee_s, x)^{p}=(\alpha^\vee_s, x)$ for all $x\in \h^*_\F$, so the expression is equal to the sum over all possible values of one variable $a_{2}$ of a polynomial of degree $1$; this sum is again by Lemma \ref{ntuplesum} equal to zero whenever $1<p-1$. So, the expression is equal to zero, as desired, whenever $(n,q)\ne (2,2)$.
\end{proof}

When $(n,q)=(2,2)$, the claim of the lemma is not true, and the irreducible module with trivial lowest weight is bigger. We will settle the case $(n,q)=(2,2)$ separately by explicit calculations. 

As explained before, studying $L_{1,c}(\mathrm{triv})$ is the same as studying the contravariant form $B$ on $M_{1,c}(\mathrm{triv})$. The following proposition tells us that the set of singular vectors from the previous proposition is large, in the sense that the quotient of $M_{1,c}(\mathrm{triv})$ by the submodule generated by them is already finite dimensional. 

\begin{cor}\label{corBzero}
Suppose $(q,n) \not= (2,2)$.  Then, the form $B_k=0$ and $L_{1,c}(\mathrm{triv})_k=0$ for all $k\ge np - n + 1$.
\end{cor}
\begin{proof}
Any degree $n(p - 1) + 1$ monomial must have one of the $n$ basis elements, say $x_i$, raised to at least the $p$-th power by the pigeonhole principle.  Then, $B(x_1^{a_1} \ldots x_n^{a_n}, y) = B(x_i^p, y') = 0$, where $y'$ is the result of the Dunkl operator being applied to $y$ successively. 
\end{proof}

We write the matrices of the form $B_k$ in the monomial basis in $x_i$ for $S\h^*$ and $y_i$ for $S\h$, both in lexicographical order. In the case of $GL_n(\F_q)$ these matrices are surprisingly simple. 

\begin{prop}\label{gln-diag}
Suppose $(q,n) \not= (2,2)$.  Then the matrices of $B_k$ are diagonal for all $k$.
\end{prop}
\begin{proof}
We will use invariance of $B$ with respect to $G$ to show that for $(a_1,\ldots, a_n)$, $(b_1,\ldots , b_n)\in \mathbb{Z}_{\ge0}^n$, such that $\sum a_i=\sum b_i=k$, if $B(x_1^{a_1} \ldots x_n^{a_n}, y_1^{b_1} \ldots y_n^{b_n})\ne 0$ then  $(a_1,\ldots , a_n)=(b_1,\ldots , b_n)$. This means that the matrices $B_k$, written in monomial basis in lexicographical order, are diagonal. 

Let $g \in G$ be a diagonal matrix with entries $\lambda_1, \ldots, \lambda_n \in \F_q^{\times}$, so that  $g \ldot y_i = \lambda_i y_i$ and $g \ldot x_i = \lambda_i^{-1} x_i$.  Then for any $(a_1,\ldots, a_n)$, $(b_1,\ldots , b_n)\in \mathbb{Z}_{\ge0}^n$, such that $\sum a_i=\sum b_i=k$, we have
\begin{align*}
B(x_1^{a_1} \ldots x_n^{a_n}, y_1^{b_1} \ldots y_n^{b_n}) & = B(g \ldot (x_1^{a_1} \ldots x_n^{a_n}), g \ldot (y_1^{b_1} \ldots y_n^{b_n})) \\
 & = B((\lambda_1^{-1} x_1)^{a_1} \ldots (\lambda_n^{-1} x_n)^{a_n}, (\lambda_1 y_1)^{b_1} \ldots (\lambda_n y_n)^{b_n}) \\
& = \lambda_1^{b_1 - a_1} \ldots \lambda_n^{b_n - a_n} B(x_1^{a_1} \ldots x_n^{a_n}, y_1^{b_1} \ldots y_n^{b_n}).
\end{align*}
So, whenever $B(x_1^{a_1} \ldots x_n^{a_n}, y_1^{b_1} \ldots y_n^{b_n}) \not= 0$, it follows that  $\lambda_1^{b_1 - a_1} \ldots \lambda_n^{b_n - a_n} = 1$ for all $\lambda_1, \ldots, \lambda_n \in \F_q^{\times}$.  Fix $i$.  Set all $\lambda_j$ where $j \ne i$ to be equal to 1, and set $\lambda_i$ to be a multiplicative generator of $\F_q^\times$.  Then, necessarily, $b_i - a_i \equiv 0 \pmod{q-1}$.  

If $q > p$ and $b_i - a_i \equiv 0 \pmod{q-1}$, then either  $a_i = b_i$ for all $i$, or there exists an index $i$ for which  $b_i - a_i \equiv 0 \pmod{q-1}$, $a_i\ne b_i$. In the second case, either $a_i$ or $b_i$ is greater than or equal to $p$, so by the previous proposition, $B(x_1^{a_1} \ldots x_n^{a_n}, y_1^{b_1} \ldots y_n^{b_n}) = 0$. This finishes the proof if $q>p$.

Now assume $q = p$, and $B(x_1^{a_1} \ldots x_n^{a_n}, y_1^{b_1} \ldots y_n^{b_n}) \ne 0$, so $a_i\equiv b_i \pmod{p-1}$ for all $i$, and $a_i,b_i<p$ for all $i$. Then either all $a_i=b_i$ as claimed, or, there exists an index $i$ such that $\{ a_i, b_i\}=\{ 0,p-1\}$. Assume without loss of generality that $i=1$, $a_1=0$, $b_1=p-1$. Using that $\sum a_j=\sum b_j=k$, there exists another index, which we can assume without loss of generality to be $2$, such that $a_2=p-1$, $b_2=0$. Now we are claiming that for any $f$ monomial in $x_3,\ldots , x_n$, any $f'$ monomial in $y_3,\ldots , y_n$, $$B(x_2^{p-1}f, y_1^{p-1}f')=0.$$ We will be working only with indices $1$ and $2$ and choosing $g\in G$ which fixes all others, so assume without loss of generality that $n=2$, $f=f'=1$. We use the invariance of $B$ with respect to $d_{1}\in G$: 
\begin{align*}
B(x_1^{p-1}, y_1^{p-1}) = B(d_1 \ldot (x_1^{p-1}), d_1 \ldot (y_1^{p-1})) = &  B((x_1 - x_2)^{p-1}, y_1^{p-1}) \\
=& B((x_1^{p-1} + \ldots + x_2^{p-1}), y_1^{p-1}) \\
=&B(x_1^{p-1} , y_1^{p-1}) + 0 + B(x_2^{p-1}, y_1^{p-1})
\end{align*}
where the terms in the middle are all zero, as their exponents do not differ by a multiple of $p - 1$.  Thus, $B(x_2^{p-1}, y_1^{p-1}) = 0$ as desired.
\end{proof}
%This proposition tells us that for a monomials $f \in S\h^*$ and $h \in S\h$, $B(f, h) \ne 0$ only if $f$ and $h$ have equivalent expressions in dual bases, e.g. $f = x_1^2 x_2$ and $h = y_1^2 y_2$.

Elements of $G\subseteq H_{t,c}(G,\h)$ have degree $0$ and preserve the graded pieces. So, every graded piece is a finite dimensional representation of $G$. This makes the following lemma useful.

\begin{lemma}\label{gln-shirreducible}
For every $i$, the space $ S^i\h^*/(\langle x_1^p, \ldots, x_n^p \rangle \cap S^i\h^*)\cong \mathrm{span}_{\Bbbk}\{x_1^{a_1}\ldots x_n^{a_n} | a_1,\ldots ,a_n<p \}$ is either zero or an irreducible $G$-representation.
\end{lemma}
\begin{proof}

This can be found in \cite{Krop}, Section $2$, Coraollary $2.6$. The result there describes certain irreducible modules $W^{\lambda}$. In the notation of \cite{Krop} and for $\lambda=(d,0,0,\ldots)$ , the the Corollary says that the modules $\mathrm{span}_{\Bbbk}\{x_1^{a_1}\ldots x_n^{a_n} | a_1,\ldots ,a_n<p \}$, when they are nonzero, are the irreducible modules $W^{(d,0,0,\ldots)}$.

\end{proof}

\begin{prop}\label{gln-multiples}
Suppose $(q,n) \not= (2,2)$.   Then, in each degree $i$, the diagonal elements of the matrix of $B_i$ are constant multiples of the same polynomial in $c_s$.
\end{prop}

\begin{proof}
By previous lemmas, every $B_i$ is a diagonal matrix, with diagonal entries polynomials $f_m$ in $c$ parametrized by $n$-tuples of integers $m=(m_1,\ldots , m_n)$ such that $\sum m_j=i$. The kernel of $B$ at specific $c$ is spanned by all monomials $x^m$ for which $f_m(c)=0$. As the kernel is a submodule of $S^i\h^*$ containing $<x_1^p,\ldots , x_n^p>\cap S^i\h^*$, by the previous lemma it can either be $<x_1^p,\ldots , x_n^p>\cap S^i\h^*$ or the whole $S^i\h^*$. In other words, all polynomials $f_m$ where one of $m_j$ is $\ge p$ are identically zero, and all others have the same roots, so they are constant multiples of the same polynomial in $c_s$.
\end{proof}

%This settles the case of generic $c$ for $(n,q)\ne (2,2)$, by implying that the maximal proper submodule $J_{1,c}(\mathrm{triv})$ is $\langle x_1^p, \ldots, x_n^p \rangle $. 

Next, we will find these polynomials  $f_m$ in each degree and calculate their zeros.

\begin{prop}\label{gln-cdep}
Suppose $(q,n) \not= (2,2)$.  Then,
\begin{itemize}
\item[a)] If $n = 2$ and $q = p$, then $B_i$ depends on the $c_s$ for $p - 1\le i \le 2p - 1$.  The diagonal entries of $B_i$ are $\Bbbk$-multiples of $c_1 + \ldots + c_{p-1} - 1$ for $i\ge p-1$ and $\Bbbk$-multiples of $(c_1 + \ldots + c_{p-1} - 1)(c_1 + 2c_2 + \ldots + (p-1)c_{p-1} + 1)$ for $i\ge p$.
\item[b)] If $n = 3$ and $q = 2$, then $B_i$ depends on the $c_s$ for $i=2,3$, and the diagonal entries are $\Bbbk$-multiples of $c_1 + 1$. 
\item[c)] Otherwise, the form $B$ doesn't depend on $c$.
\end{itemize}
\end{prop}

\begin{proof}
The matrices $B_i$ are diagonal, with all diagonal entries being constant multiples of the same polynomial. Our strategy is to compute one nonzero diagonal entry.  

First, we will show claim c). It is sufficient to show that the Dunkl operators on the quotient $M_{1,c}(\tau)/\Ker B$ are independent of $c$. As in the proof of Proposition \ref{gln-kernel}, we compute the part of the Dunkl operator associated to the conjugacy class $C_{\lambda}$ with eigenvalue $\lambda$, and claim that for any monomial $f \in S^i\h^*$, the part of $D_y(f)$ which is the coefficient of $c_{\lambda}$,
$$-\sum_{s \in C_{\lambda}} (y, \alpha_s)\frac{1}{\alpha_s} (1 - s) \ldot f,$$ is in $\Ker B.$

Using Proposition \ref{reflequiv}, we can write this sum over nonzero $\alpha \in \h^*$ and $\alpha^\vee \in \h$, such that $(\alpha, \alpha^\vee) = 1 - \lambda$. Writing it as consecutive sums over $\alpha$ and then over $\alpha^{\vee}$, it is enough to show that for any choice of $\alpha$, the inner sum, over all $\alpha^{\vee}$ such that $(\alpha, \alpha^\vee) = 1 - \lambda$, is contained in $\left<x_1^p,\ldots, x_n^p  \right>$. As in the previous calculations, we fix $\alpha$, and change basis  of $\h^*$ to $x_1',\ldots , x_n'$ so that $x_1'=\alpha$. Let the dual basis of $\h$ be $y_1',\ldots , y_n'$. The inner sum, with vectors written in the basis $y_i'$, is then over  $\alpha^{\vee}\in A_\lambda := \{((1 - \lambda),a_2, \ldots,a_n)\ne 0 \mid a_2, \ldots, a_n \in \F_q\}$.  By Proposition \ref{reflequiv}, the reflection  $s$ corresponding to $\alpha\otimes \alpha^{\vee}$, $\alpha=(1,0,\ldots, 0)$, $\alpha^{\vee}=((1 - \lambda),a_2, \ldots,a_n)$, acts on $\h^*$ as
$$s\ldot x_1'=\lambda x_1'$$
$$s \ldot x_i' = x_i' - a_i x_1', \,\,\, i > 1$$

In addition to fixing $\alpha$, let us also factor out the constant $-(y,\alpha)$. The inner sum for fixed $\alpha$ is then $\sum_{\alpha^{\vee}\in A_\lambda}\frac{1}{\alpha}(1-s)\ldot f$. The set $A_{\lambda}$ is parametrized by $(a_2,\ldots , a_n)\in \F_q^{n-1}$ if $\lambda\ne 1$, and by $(a_2,\ldots , a_n)\ne 0 \in \F_q^{n-1}$ if $\lambda=1$. However, if $\lambda=1$, the summand corresponding to $(a_2,\ldots , a_n)= 0$ is $0$, so we can assume the sum is over all $(a_2,\ldots, a_n)\in \F_q^{n-1}$ in both cases. 

The inner sum  we are calculating is equal to
\begin{equation}\tag{$\star$}
\sum_{(a_2, \ldots, a_n) \in \F_q^{n-1}} \frac{1}{x_1'} \Big( f(x_1', \ldots, x_n') - f(\lambda x_1', x_2' - a_2 x_1', \ldots, x_n' - a_n x_1') \Big).
\end{equation}

It is enough to calculate it for $f$ of the form $f=x_{1}'^{b_1}\ldots x_{n}'^{b_n}$, $b_i<p$. For such $f$,
$$(\star) =\sum_{(a_2, \ldots, a_n) \in \F_q^{n-1}} \frac{-1}{x_1'} \Big( (\lambda x_{1}')^{b_1} (x_{2}'-a_{2}x_{1}')^{b_2} \ldots (x_{n}'-a_{n}x'_{1})^{b_n} -x_{1}'^{b_1}\ldots x_{n}'^{b_n} \Big) =$$
$$=\sum_{(a_2, \ldots, a_n) \in \F_q^{n-1}} \sum_{i_2,\ldots i_n} - {b_{2} \choose i_{2}}\ldots {b_{n} \choose i_{n}} \lambda^{b_{1}}(-a_{2})^{i_{2}}\ldots (-a_{n})^{i_{n}} x_{1}'^{b_{1}+i_{2}+\ldots i_{n}-1}x_{2}'^{b_{2}-i_{2}}\ldots x_{n}'^{b_{n}-i_{n}}$$
the last sum being over all $0\le i_{j}\le b_{j}$ such that not all  $i_{j}$ are $0$ at the same time. The coefficient of each monomial in $x_{i}$ can be seen as a monomial of degree $i_{2}+\ldots i_{n}$ in variables $a_{i}$, so when we sum it over all $(a_{2},\ldots a_{n}) \in \F_{q}^{n-1}$ to get the sum ($\star$), we can use Lemma \ref{ntuplesum} to conclude ($\star$) is $0$ whenever the degree of all polynomials appearing is small enough, more precisely, whenever there exists an index $j$ such that 
$i_{j}<q-1.$ As $i_j\le b_{j}<p$ for all $j$, this only fails when $p=q$ and $b_{2}=b_{3}=\ldots=b_{n}=p-1$. In other words, this proves c) whenever $q\ne p$. 

Now assume $q=p$. By the above argument, the only monomials $f$ for which ($\star$) is not yet known to be zero are the ones of the form $f = x_1'^b x_2'^{p-1} \cdots x_n'^{p-1}$. For such $f$, the sum ($\star$) is by the above argument equal to 
$$ (-1)^{(p-1)(n-1)+1} \sum_{(a_2, \ldots, a_n) \in \F_q^{n-1}} \lambda ^{b}(a_{2})^{p-1}\ldots (a_{n})^{p-1} x_{1}'^{b+(p-1)(n-1)-1}$$

While this is not $0$, the monomial $x_{1}'^{b+(p-1)(n-1)}$ is by Lemma \ref{gln-kernel} in $\Ker B$ whenever it has degree at least $p$, meaning whenever $$b+(p-1)(n-1)-1\ge p.$$ 

If $n = 2$, this condition is $b \ge 2$, which is not satisfied only when $b = 0, 1$.  Thus, for $n=2$, $q=p$, the diagonal matrices $B_{i}$ do not depend on $c$ in degrees $i<p-1$, their entries are multiples of some polynomial in $c$ in degree $p-1$, and some other polynomial (divisible by the first one) in degrees $p$ and higher. Finally, by corollary \ref{corBzero}, all matrices $B_{i}$ become zero at degrees $2p-1$ and higher.

For $n = 3$, this condition is $b+2(p-1)-1 \ge p$, which is equivalent to $p + b \ge 3$. This will not be satisfied only if $p = 2$ and $b = 0$. So, for $n=3$, $p=2$, the matrices $B_0$ and $B_{1}$ do not depend on $c$, the diagonal entries of $B_{i}$ are multiples of the same polynomial in $c$ for $i=2,3$, and $B_{4}=0$. For $n=3$ and all other $p$, the matrices $B_i$ do not depend on $c$.

For $n > 3$, the inequality $b+(p-1)(n-1)-1\ge p$ is always satisfied, so there is never a dependence of matrices $B_{i}$ on the parameters $c$. 

This finishes the proof of (c) and describes the cases in a) and b) for which there might be dependence on parameters $c$.  To finish the proof, it remains to find specific polynomials in cases: (a) $n=2,q=p$, degrees $p-1$ and $p$ and; (b)$n=3,q=2$, degrees $2$ and $3$.

Next, we prove (a). Let $n = 2$ and $q = p$.  We need to compute one nonzero entry of the matrix $B_{p-1}$ and one nonzero entry of $B_{p}$.

To compute the polynomial in degree $p-1$, by Proposition \ref{gln-multiples}, it suffices to compute $B(x_1^{p-1}, y_1^{p-1}) = B(D_{y_{1}}(x_1^{p-1}), y_1^{p-2})$.  For that, by Proposition \ref{gln-diag}, it suffices to show that the coefficient of $x_1^{p-2}$ in $D_{y_1}(x_1^{p-1})$ is a constant multiple of $c_1 + \ldots + c_{p-1} - 1$. Compute
\begin{equation*}
D_{y_1}(x_1^{p-1})  = \partial_{y_1}(x_1^{p-1}) - \sum_{C} c_s \sum_{s \in C} (\alpha_s, y_1) \frac{1}{\alpha_s} ( x_1^{p-1}-(s\ldot x_{1})^{p-1}). \\
\end{equation*}
The coefficient of $x_1^{p-2}$ in $\partial_{y_1}(x_1^{p-1})=(p-1)x_{1}^{p-2}$ is $p-1=-1$, so it suffices to show that for each conjugacy class $C$ the coefficient of $x_1^{p-2}$ in
$$\sum_{s \in C} (\alpha_{s}, y_1) \frac{1}{\alpha_s} ((s \ldot x_1)^{p-1}-x_1^{p-1} )$$
is $1$.  Using $\binom{p-1}{i}=(-1)^i$, we can write this as
$$\sum_{s \in C} (\alpha, y_1) \frac{1}{\alpha_s} ((x_1 - (x_1, \alpha_s^\vee)\alpha_s)^{p-1}-x_1^{p-1}) =$$
$$= \sum_{s \in C} (\alpha_s, y_1)\sum_{i=1}^{p-1}  (\alpha^\vee_{s}, x_1)^i \alpha_s^{i-1} x_1^{p-1-i}.$$
The bases $x_{i}$ and $y_{i}$ are dual, so $\alpha_{s}=(\alpha_{s},y_1)x_1+(\alpha_{s},y_2)x_2$, and the coefficient of  $x_1^{i-1}$ in $\alpha_{s}^{i-1}$ is $(\alpha_s, y_1)^{i-1}$. Thus, the coefficient of $x_1^{p-2}$ in the above sum is
$$\sum_{s \in C} (y_1, \alpha_s) \sum_{i=1}^{p-1} (\alpha^\vee_s, x_1)^i (y_1, \alpha_s)^{i-1}$$ which can be written as
$$ \sum_{s \in C} \left( ((\alpha_s^\vee, x_1)(\alpha_s, y_1) - 1)^{p-1} - 1\right).$$
Each term $((\alpha_s^\vee, x_1)(\alpha_s, y_1) - 1)^{p-1} - 1$ is nonzero if and only if $(\alpha_s^\vee, x_1)(\alpha_s,y_1) = 1$, in which case it is $-1$.  There are $p-1$ choices of the first coordinates of $\alpha_s, \alpha^\vee_s$ that make their product 1.  The product of the second coordinate must now be $(\alpha, \alpha^\vee) - 1$.  This is nonzero, so there are $p - 1$ choices for the second coordinates. Hence, the sum is $(p-1)^2(-1) = -1$, as desired.  Note that this term will appear as a multiplicative factor in higher degrees, since the matrices of $B$ are defined inductively. This proves the claim for degree $p-1$.

Now let us consider degree $p$, with $n=2$ and $q=p$.  We will calculate the value of $B(x_1^{p-1} x_2, y_1^{p-1} y_2)=B(D_{y_2}(x_1^{p-1} x_2),y_1^{p-1})$, equal to the product of the coefficient of $x_1^{p-1}$ in $D_{y_2}(x_1^{p-1} x_2)$ with  $B(x_1^{p-1},y_1^{p-1})$. We proved that $B(x_1^{p-1},y_1^{p-1})$ is a constant multiple of $c_1+\ldots +c_{p-1}-1$, so we are now calculating the coefficient of $x_1^{p-1}$ in 
 \begin{align*}
D_{y_2}(x_1^{p-1} x_2) & = \partial_{y_2}(x_1^{p-1}x_2) - \sum_{\lambda} c_\lambda \sum_{s \in C\lambda} (\alpha_s, y_2) \frac{1}{\alpha_s} (1 - s) \ldot (x_1^{p-1}x_2) \\
 & = x_1^{p-1} - \sum_\lambda c_\lambda \sum_{s \in C_\lambda} (\alpha_s, y_2) \frac{1}{\alpha_s} (x_1^{p-1}x_2 - (s \ldot x_1)^{p-1}(s \ldot x_2)).
\end{align*}

We now use the parametrization of conjugacy classes $C_{\lambda}$ by $\alpha\otimes \alpha^{\vee}$ from Lemma \ref{reflequiv} and Example \ref{n2refl}:
$$\lambda\ne 1:\,\,\,\, C_{\lambda}\leftrightarrow \{\left[\begin{array}{c} 1 \\ b \end{array} \right] \otimes \left[\begin{array}{c} 1-\lambda-bd \\ d \end{array} \right]  | b,d\in \F_p  \} \cup \{ \left[\begin{array}{c} 0 \\ 1 \end{array} \right] \otimes \left[\begin{array}{c} a \\ 1-\lambda \end{array} \right]  | a \in \F_p \} $$
$$\lambda= 1:\,\,\,\, C_{1} \leftrightarrow\{\left[\begin{array}{c} 1 \\ b \end{array} \right] \otimes \left[\begin{array}{c} -bd \\ d \end{array} \right]  | b,d\in \F_p, d\ne 0  \} \cup \{ \left[\begin{array}{c} 0 \\ 1 \end{array} \right] \otimes \left[\begin{array}{c} a \\ 0 \end{array} \right]  | a \in \F_p, a\ne 0 \}.$$

We are calculating the coefficient of $x_1^{p-1}$ in 
$$ x_1^{p-1} - \sum_\lambda c_\lambda\sum_{\alpha\otimes\alpha^{\vee}\leftrightarrow C_{\lambda}}(\alpha, y_2) \frac{1}{\alpha} \left(x_1^{p-1}x_2 - (x_1-(\alpha^{\vee}, x_1)\alpha)^{p-1}(x_2-(\alpha^{\vee},x_2)\alpha)\right)=$$
$$=x_1^{p-1} - \sum_\lambda c_\lambda\sum_{\alpha\otimes\alpha^{\vee}\leftrightarrow C_{\lambda}}(\alpha, y_2) \frac{1}{\alpha} \left(x_1^{p-1}x_2 - (x_1-(\alpha^{\vee}, x_1)\alpha)^{p-1}x_2\right)-$$
$$ - \sum_\lambda c_\lambda\sum_{\alpha\otimes\alpha^{\vee}\leftrightarrow C_{\lambda}}(\alpha, y_2) \frac{1}{\alpha} \left((x_1-(\alpha^{\vee}, x_1)\alpha)^{p-1}(\alpha^{\vee},x_2)\alpha\right).$$
The term $x_1^{p-1}-(x_1-(\alpha^{\vee}, x_1)\alpha)^{p-1}$ is divisible by $\alpha$, so $$\frac{1}{\alpha} \left(x_1^{p-1}x_2 - (x_1-(\alpha^{\vee}, x_1)\alpha)^{p-1}(x_2-(\alpha^{\vee},x_2)\alpha)\right),$$ written in a monomial basis in $x_1$ and $x_2$, is divisible by $x_2$. These terms can be disregarded when calculating the coefficient of $x_1^{p-1}$ in the above sum. 

Let $\alpha_{b}=x_1+bx_2$. We are calculating the coefficient of $x_1^{p-1}$ in 
$$x_1^{p-1}  - \sum_\lambda c_\lambda\sum_{\alpha\otimes\alpha^{\vee}\leftrightarrow C_{\lambda}}(\alpha, y_2) \frac{1}{\alpha} \left((x_1-(\alpha^{\vee}, x_1)\alpha)^{p-1}(\alpha^{\vee},x_2)\alpha\right)=$$
$$=x_1^{p-1}  - \sum_\lambda c_\lambda \Big(
\sum_{b,d} b \frac{1}{\alpha_b} \left((x_1-(1-\lambda-bd)\alpha_b)^{p-1}d \alpha_b\right) +
\sum_{a}\frac{1}{x_2} \left((x_1-ax_2)^{p-1}(1-\lambda)x_2\right) 
\Big)$$
$$=x_1^{p-1}  - \sum_\lambda c_\lambda \Big(
\sum_{b,d} bd\sum_{i=0}^{p-1}(1-\lambda-bd)^ix_1^{p-1-i}(x_1+bx_2)^i+(1-\lambda)\sum_{a}(x_1-ax_2)^{p-1}
\Big)$$
Here, the sum over is over all $b,d\in \F_p$ and over all $a\in \F_p$ if $\lambda\ne 1$, and over $a,b,d\in \F_p$, $d,a\ne 0$ if $\lambda=1$. The coefficient of $x_1^{p-1}$ is:
$$1- \sum_\lambda c_\lambda \Big(
\sum_{b,d} bd\sum_{i=0}^{p-1}(1-\lambda-bd)^i +(1-\lambda)\sum_{a}1 \Big)=$$
$$1-\sum_{\lambda\ne 1} c_\lambda \Big(
\sum_{b,d} (bd)^{p-1}(-\lambda) \Big) -c_1 \Big( \sum_{b,d} (bd)^{p-1} \Big)=1+\sum_{\lambda} \lambda c_\lambda.$$

This ends the proof of (a).

To prove (b), $n = 3$, $q = 2$, we computed the matrices $B_{i}$ explicitly.
\end{proof}

The combination of these results and the explicit computations in case $(n,q)=(2,2)$ gives us the main theorem of this section:
\begin{thm}\label{main}
Let $\Bbbk$ be an algebraically closed field of characteristic $p$.  Let $G = GL_n(\F_q)$ for $q = p^r$ and $n\ge 2$. The  following is a complete classification of characters of $L_{1,c}(\triv)$ for all values of $c$:\\

\hspace{-0.5cm}\begin{tabular}{|c|c|c|c|}
\hline
$(q,n)$ & $c$ & $\chi_{L_{1,c}(\mathrm{triv})}(z)$ & $\mathrm{Hilb}_{L_{1,c}(\mathrm{triv})}(z)$ \\ \hline \hline
$(q,n) \ne (2,2)$ & generic & $\sum_{i\ge 0}(\left[ S^i\h^*/\left<x_1^p,\ldots , x_n^p \right>\cap S^i\h^*\right])z^i$ & $\left(\displaystyle\frac{1 - z^p}{1 - z}\right)^n$ \\ \hline
$(2,2)$ & generic & $\left(\sum_{i\ge 0}\left[S^i\h^*\right] z^i\right)(1-z^4)(1-z^6)$ & $\displaystyle\frac{(1 - z^4)(1 - z^6)}{(1 - z)^2}$ \\ \hline
\hline
%$(q,n)$ & $c$ & $\chi_{L_{1,c}(\tau)}(z)$& $\mathrm{Hilb}_{L_{1,c}(\tau)}(z)$ \\ \hline
$ (2,3)$ & $c=1$ & $\displaystyle{\left[\mathrm{triv}\right]+\left[\h^*\right]z}$ & $1 + 3z$ \\ \hline
$(p, 2), p \ne 2$ & $\displaystyle{\sum_{i=1}^{p-1}c_i=1}$ & $\displaystyle{\sum_{i=0}^{p-2}\left[ S^i\h^*\right]z^i}$ & $\displaystyle{\sum_{i=0}^{p-2}(i+1)z^i }$ \\ \hline
$(p, 2), p \ne 2$ & $\displaystyle{\sum_{i=1}^{p-1}c_i\ne1, \sum_{i=1}^{p-1}ic_i=-1}$ & $\displaystyle{\sum_{i=0}^{p-1}\left[ S^i\h^*\right]z^i}$ & $\displaystyle{\sum_{i=0}^{p-1}(i+1)z^i } $ \\ \hline
$(2,2)$ & $c=1$ & $\displaystyle{\left[\mathrm{triv}\right]}$ & $1$ \\ \hline
$ (2,2)$ & $c=0$ & $\displaystyle{\left[\mathrm{triv}\right]+\left[\h^*\right]z+\left[\mathrm{triv}\right]z^2}$ & $1+2z+z^2$ \\ \hline

\end{tabular}

%\begin{tabular}{|c|c|c|cl}
%\hline
%& $c=1$ & $1 + 3z$ \\ \cline{2-3}
%$(q,n) = (2,3)$ & otherwise & $\left(\frac{1 - z^2}{1 - z}\right)^3$ \\ \hline
%$(q,n) = (p^r, 2), r > 1$ & all cases & $\left(\displaystyle\frac{1 - z^p}{1 - z}\right)^2$ \\ \hline
%& $c_1 + \cdots + c_{p-1} = 1$ & $1 + 2z + \cdots + (p-1)z^{p-2}$ \\  \cline{2-3}
%& $c_1 + \cdots + c_{p-1} \ne 1$ and  &  \\
%$(q,n) = (p, 2), p \ne 2$ & $c_1 + 2c_2 + \cdots + (p-1)c_{p-1} = -1$ & $1 + 2z + \cdots + pz^{p-1}$\\  \cline{2-3}
%& otherwise & $\left(\displaystyle\frac{1 - z^p}{1 - z}\right)^2$ \\ \hline
%& $c = 1$ & 1 \\  \cline{2-3}
%$(q,n) = (2,2)$ & $c = 0$ & $1 + 2z + z^2$ \\  \cline{2-3}
%& otherwise & $\displaystyle\frac{(1 - z^4)(1 - z^6)}{(1 - z)^2}$\\ \hline
%\end{tabular}
\end{thm}
\begin{proof}
For $(n,q)=(2,2)$, the matrices $B_i$ of the form $B$ can be computed explicitly, and one can see that they are $0$ starting in degree $1$ when $c=1$ and starting in degree $3$ when $c=0$. For all other $c$, they are full rank on $N_{1,c}(\mathrm{triv})$, so $L_{1,c}(\mathrm{triv})=N_{1,c}(\mathrm{triv})$ is the quotient of the Verma module by squares of the invariants, which are in degrees $4$ and $6$.

For $(q,n)\ne(2,2)$ and generic $c$, we saw in Proposition \ref{gln-kernel} that $J_0'(\triv)$ contains $\langle x_1^p, \ldots, x_n^p \rangle$, so by Proposition \ref{reducedchar} the reduced module $R_{1,c}(\triv)$ is a quotient of the trivial module. From this it follows that for generic
$c$,  $J_{1,c}(\triv)=\langle x_1^p, \ldots, x_n^p \rangle$ and that the character of $L_{1,c}(\triv)$ for generic $c$ and $(q,n)\ne(2,2)$ is as stated above.

The characters for special $c$ are computed by looking at the roots of polynomials on the diagonal in $B_i$, and are computed directly from Proposition \ref{gln-cdep}.  
\end{proof}

\begin{rem}
Notice that for $n,p,r$ large enough, the character does not depend on $c$ at all. This never happens in characteristic zero.  
\end{rem}

\begin{rem}
Notice that the claims from Remarks \ref{conj00} and \ref{conj02} and Question \ref{conj01} hold in the case of $G=GL_{n}(\F_q)$. Namely, by observing the characters one can see that all $L_{t,c}(\triv)$ for generic $c$ have one dimensional top degree and are thus Frobenius; that for $h_1$ the reduced Hilbert series of $L_{1,c}(\triv)$ at generic $c$, $h_1(1)$ is either $|G|$ (in the case of $(q,n)=(2,2)$, when they are both $6$) or $h_1(1)=1$ (in all other cases), and that for $h_0$ the Hilbert series of $L_{0,c}(\triv)$ at generic $c$, the equality $h_0=h_1$ always holds. 
\end{rem}

\section{Description of $L_{t,c}(\triv)$ for $SL_n(\F_{p^r})$}\label{slnsec}

In this section we explore category $\mathcal{O}$ for the rational Cherednik algebra associated to the special linear group over a finite field. We start with some preliminary facts about relations between rational Cherednik algebras associated to some group and to its normal subgroup, and by looking more carefully into conjugacy classes of reflections in $SL_n(\F_{p^r})$.

\subsection{Normal subgroups of reflection groups}
Let $G\subseteq GL(\h)$ be any reflection group, and assume $N \subset G$ is a normal subgroup with a property that two reflections in $N$ are conjugate in $N$ if and only if they are conjugate in $G$. Let $c$ be a $\Bbbk$-valued conjugation invariant function on reflections of $N$, and extend $c$ to all reflections in $G$ by defining it to be zero on reflections which are not in $N$. Then one can consider the rational Cherednik algebra $H_{t,c}(N,\h)$ as a subalgebra of $H_{t,c}(G,\h)$; it has fewer generators and the same relations. 

Let $\tau$ be an irreducible representation of $G$, and assume it is irreducible as a representation $\tau|_{N}$ of $N$. Consider two representations of $H_{t,c}(N,\h)$: the irreducible representation $L_{t,c}(\tau|_{N})=L_{t,c}(N,\h,\tau|_{N})$, and the irreducible representation $L_{t,c}(\tau)=L_{1,c}(G,\h,\tau)$ of $H_{t,c}(G,\h)$ restricted to $H_{t,c}(N,\h)$. 

\begin{lemma}\label{normalsubgp}
As representations of $H_{1,c}(N,\h)$, $L_{1,c}(\tau|_{N})\cong L_{1,c}(\tau)|_{H_{1,c}(N,\h)}$.
\end{lemma}

\begin{proof}
The reflections in $N$ are a subset of reflections in $G$. Because $N$ is normal in $G$, every conjugacy class in $G$ is either contained in $N$ or does not intersect it. By the assumption, two reflections in $N$ which are conjugate in $G$ are also conjugate in $N$, so conjugacy classes in $N$ are a subset of conjugacy classes in $G$. 

The Verma modules $M_{t,c}(G,\h,\tau)$ and $M_{t,c}(N,\h,\tau|_N)$ are both naturally isomorphic to $S\h^*\otimes \tau$ as vector spaces, and this induces their natural isomorphism as  $H_{t,c}(N,\h)$ representations. The modules $L_{t,c}(\tau)$ and $L_{t,c}(\tau|_N)$ are their quotients by the kernel of the contravariant form, which is controlled by Dunkl operators. Because of the discussion of conjugacy classes in $N$ and $G$ and because of the definition of $c$, the Dunkl operators are the same for $H_{t,c}(N,\h)$ and $H_{t,c}(G,\h)$.
\end{proof}

One could define Verma modules, baby Verma modules and their quotients by the kernel of the  contravariant form (chosen so that it is non-degenerate on the lowest weights) even in cases when the lowest weight is not irreducible as a representation of the reflection group. In that case, an analogous lemma  is that the composition series of $\tau$ as a representation of $N$ is the same as the composition series of $L_{t,c}(\tau)$ as a representation of $H_{t,c}(N,\h)$. We will not need this here. 

\subsection{Conjugacy classes of reflections in $SL_n(\F_q)$}

In this section we will study $G = SL_n(\F_q)$ for $q = p^r$ a prime power and $n>1$.  As before, let $\h = \Bbbk^n$, $\Bbbk = \overline{\F}_p$, and $\tau$ be the trivial representation. Further, let $Q$ be the set of nonzero squares in $\F_q$ and $R$ be the set of non-squares.

All reflections in $SL_n(\F_q)$ are unipotent and conjugate in $GL_{n}(\F_q)$ to $d_{1}$. It is easy to see that $SL_{n}(\F_q)$ is generated by them. The group $SL_n(\F_q)$ is a normal subgroup of  $GL_n(F_q)$, and it contains all the reflections from the unipotent conjugacy class in $GL_{n}(\F_q)$. However, the second condition from the above discussion, that two reflections are conjugate in $GL_{n}(\F_q)$ if and only if they are conjugate in $SL_{n}(\F_q)$, is not satisfied for all $n,q$. For example, $$\left[ \begin{array}{cc} 1 & -1\\ 0 & 1 \end{array}\right]= \left[ \begin{array}{cc} -1 & 0\\ 0 & 1 \end{array}\right] \left[ \begin{array}{cc} 1 & 1\\ 0 & 1 \end{array}\right] \left[ \begin{array}{cc} -1 & 0\\ 0 & 1 \end{array}\right] \textrm{ is not conjugate in }SL_2(\F_3) \textrm{ to } \left[ \begin{array}{cc} 1 & 1\\ 0 & 1 \end{array}\right].$$

\begin{prop}\label{sln-nosplit}
Let $q = p^r$ be a prime power. If $n \geq 3$, or $p = 2$, then two reflections are conjugate in $SL_n(\F_q)$ if and only if they are conjugate in $GL_n(\F_q)$, and there is one conjugacy class of reflections in $SL_n(\F_q)$. Otherwise (for $n=2$ and $q=p^r$, $p\ne 2$),  there are two conjugacy classes of unipotent reflections in $SL_n(\F_q)$:
$$C_{Q}=\{ gd_1g^{-1} | g\in GL_{2}(\F_q), \det(g) \textit{ is a square} \}$$ and 
$$C_{R}=\{ gd_1g^{-1} | g\in GL_{2}(\F_q), \det(g) \textit{ is not a square} \}.$$

\end{prop}
\begin{proof}
Direct computation, using the description of conjugacy classes in $GL_n(\F_q)$ and the fact that two reflections are conjugate in  $SL_n(\F_q)$ if and only if they are conjugate in $GL_n(\F_q)$ by some element whose determinant is in $Q$ (the set of squares in $\F_q$).

\end{proof}

If $n\ge 3$ or $p=2$, then the only conjugacy class in $SL_n(\F_q)$ is equal to the conjugacy class $C_1$ of $GL_n(\F_q)$. Let us call this class $C$, and the constant associated to it $c$. If $n=2$ and $p\ne 2$, to the two conjugacy classes $C_R$ and $C_Q$ we will associate parameters  $c_R$ and $c_Q$. Note that  $C_R \cup C_Q=C_1$. In the case of only one conjugacy class, we will use Lemma \ref{normalsubgp} to transfer character formulas for rational Cherednik algebras associated to $GL_n(\F_q)$ to character formulas for rational Cherednik algebras associated to $SL_n(\F_q)$. In the case of two conjugacy classes, we will have to do more computations to get character formulas. Let us first look more closely into the case of two conjugacy classes.

\begin{lemma}\label{sln-qrcc}
Let $n=2$, $q=p^r$, and $p\ne 2$. Let $\gamma \in R$ be an arbitrary non-square in $\F_q$.  Let $s$ be a reflection in $SL_2(\F_q)$.  Then, $s$ and $\gamma s - (\gamma - 1)$ are in different conjugacy classes.
\end{lemma}
\begin{proof}
The proof follows from Proposition \ref{sln-nosplit}. The map $F_{\gamma}: s\mapsto \gamma s - (\gamma - 1)$ maps reflections to reflections, and its inverse is $F_{\gamma^{-1}}$. So, it is enough to show it maps $s\in C_{Q}$ to an element of $C_{R}$.

Remember that $d_{\gamma}$, $\gamma\ne 0,1$ is the diagonal matrix with diagonal entries $\gamma^{-1}, 1$, and $d_1$ is a matrix with all generalized eigenvalues equal to $1$ and one Jordan block of size $2$. For $s=d_{1}$, we have $F_{\gamma}(d_{1})=d_{\gamma}^{-1}d_{1}d_{\gamma} \in C_{R}$. If $s$ is a conjugate in $SL_{2}(\F_{q})$ to $d_{1}$, say $s=hd_1h^{-1}$, then $F_{\gamma}(s)=hd_{\gamma}^{-1}d_{1}d_{\gamma}h^{-1}=(hd_{\gamma}^{-1}h^{-1})s(hd_{\gamma}^{-1}h^{-1})^{-1}\in C_{R}$.
\end{proof}

The following lemma is useful in computations, and is a stronger version of Lemma \ref{ntuplesum}. 

\begin{lemma}\label{quadressum}
If $d \equiv 0 \pmod{q-1}$, then $\sum_{i \in Q} i^d = \sum_{i \in R} i^d = \frac{q-1}{2}$.  If $d \equiv \frac{q-1}{2} \pmod{q-1}$, then $\sum_{i \in Q} i^d = -\sum_{i \in R} i^d = \frac{q-1}{2}$.  Otherwise, $\sum_{i \in Q} i^d = \sum_{i \in R} i^d = 0$.
\end{lemma}
\begin{proof}
For this proof, let $S_Q = \sum_{i \in Q} i^d$ and $S_R = \sum_{i \in R} i^d$.  Suppose $d \equiv 0 \pmod{q-1}$.  Then, $i^d=1$ for all nonzero $i\in \F_q$,  and $S_R = S_Q = |Q| = |R| = \frac{q-1}{2}$.  Suppose $d \equiv \frac{q-1}{2} \pmod{q-1}$. In that case, if $i \in Q$, $i^d = 1$, and if $i \in R$, then $i^d = -1$. Thus, $S_Q = -S_R = \frac{q-1}{2}$. Suppose neither holds.  For any $a \in Q$ it is easy to see that $aR=R$ and $aQ=Q$, so $S_R = a^d S_R$ and $S_Q = a^d S_Q$. If $a$ is a multiplicative generator of the cyclic multiplicative group $\F_q^{\times}$, then $1$ and  $a^{d}$ are different elements of $Q$, so $(1-a^{d})S_{Q}=(1-a^{d})S_{R}=0$ implies that $S_R = S_Q = 0$.
\end{proof}

Next, we parametrize reflections in each conjugacy class. Remember the notation from Proposition \ref{reflequiv}: unipotent reflections are identified with all $\alpha\otimes \alpha^{\vee} \in \h^*\otimes \h$ such that $(\alpha, \alpha^{\vee})=0,$ in such a way that the action of  a reflection $s$ on $x\in \h^*$ and $y\in \h$ is
$$s.x=x-(\alpha^{\vee},x)\alpha$$
$$s.y=y+(\alpha,y)\alpha^{\vee}.$$

\begin{lemma}\label{sln-param}
The conjugacy classes $C_Q$ and $C_R$ of reflections in $SL_2(\F_{p^r})$, $p>2$, are parametrized through $\alpha\otimes \alpha^{\vee}$ as
$$C_{Q}=\left\{ \gamma \left[ \begin{array}{c} 1 \\ a \end{array} \right] \otimes \left[ \begin{array}{c} a \\ -1 \end{array} \right] \mid a \in \F_q, \, \gamma \in Q\right\} \bigcup \left\{ \gamma \left[ \begin{array}{c} 0\\1 \end{array}\right] \otimes \left[ \begin{array}{c} 1 \\0 \end{array} \right] \mid \gamma \in Q \right\},$$
$$C_{R}=\left\{ \gamma \left[ \begin{array}{c} 1 \\ a \end{array} \right] \otimes \left[ \begin{array}{c} a \\ -1 \end{array} \right] \mid a \in \F_q, \, \gamma \in R \right\} \bigcup \left\{ \gamma \left[ \begin{array}{c} 0\\1 \end{array}\right] \otimes \left[ \begin{array}{c} 1 \\0 \end{array} \right] \mid \gamma \in R \right\}.$$
\end{lemma}
\begin{proof}
The proof is straightforward. The reflection $d_{1}$ is identified with $\left[ \begin{array}{c} 0 \\ 1 \end{array} \right] \otimes \left[ \begin{array}{c} 1 \\ 0 \end{array} \right] $. Conjugating it by $g = \left[ \begin{array}{cc} a&b\\c&d\end{array} \right] \in GL_{2}(\F_q)$ gives us a reflection $gd_1g^{-1}$ identified with $$\frac{1}{\det g} \left[ \begin{array}{c} -c \\ a \end{array} \right] \otimes \left[ \begin{array}{c} a \\ c \end{array} \right].$$
As $g$ can always be scaled by an element of the centralizer of $d_1$ so that either $c=-1$ or $c=0$ and $a=1$, and that $gd_1g^{-1}$ is in $C_Q$ or $C_R$ depending whether $\det g$ is in $Q$ or $R$, the description follows. 
\end{proof}

\subsection{Description of $L_{0,c}(\triv)$ for $SL_n(\F_{p^r})$}

\begin{thm}\label{slnt0} Characters of the irreducible modules  $L_{0,c}(\triv)$ for the rational Cherednik algebra $H_{0,c}(SL_n(\F_q,\h))$ are:

\begin{tabular}{|c|c|c|c|}
\hline
$(q,n)$ & $c$ & $\chi_{L_{0,c}(\triv)}(z)$ & $\mathrm{Hilb}_{L_{0,c}(\triv)}(z)$ \\ \hline \hline
$(q,n)\ne (2,2)$ & any & $[\triv]$ & $1$ \\ \hline
$(2,2)$ & $0$ & $[\triv]$ & $1$ \\ \hline
$(2,2)$ & $c\ne0$ & $[\triv]+[\h^*]z+([S^2\h^*]-[\triv])z^2+$ & $1+2z+2z^2+z^3$\\ 
& & $+([S^3\h^*]-[\h^*]-[\triv])z^3$ & \\ \hline
\end{tabular}
\end{thm}
\begin{proof}
The group $SL_{n}(\F_q)$ is a normal subgroup of $GL_n(F_q)$, and for $n\ge 3$ or $p=2$, by Proposition \ref{sln-nosplit} it satisfies the conditions of Lemma \ref{normalsubgp}. Thus, in those cases, irreducible representations of $H_{t,c}(SL_{n}(\F_q),\h)$ have the same characters as irreducible representations of $H_{t,c}(GL_{n}(\F_q),\h)$, where $c$ is extended to conjugacy classes of reflections in $GL_{n}(\F_q)-SL_{n}(\F_q) $ by zero. So, in those cases we can deduce character formulas for $L_{0,c}(\triv)$ from Theorem \ref{glnt0}. 

The remaining case is $SL_2(\F_q)$, $p\ne 2$, for which we claim that $D_y(x)=0$ for all $x$ and $y$, and so the character is trivial. We have 
$$D_y(x)=-c_R\sum_{\alpha\otimes \alpha^{\vee}\in C_R}(y,\alpha)(x,\alpha^{\vee})-c_Q\sum_{\alpha\otimes \alpha^{\vee}\in C_Q}(y,\alpha)(x,\alpha^{\vee}), $$ so it is enough to show that for $T=Q$ or $T=R$, $$\sum_{\alpha\otimes \alpha^{\vee}\in C_T}\alpha\otimes \alpha^{\vee}$$ is zero. 

We know from the proof of Theorem \ref{glnt0} that $\sum_{\alpha\otimes \alpha^{\vee}\in C_{R\cup Q}}\alpha\otimes \alpha^{\vee}=0$, so it is enough to prove that the sum over $C_Q$ is zero. For this, let us calculate, using parametrization from Lemma \ref{sln-param}:
\begin{eqnarray*} \sum_{\alpha\otimes \alpha^{\vee}\in C_Q}\alpha\otimes \alpha^{\vee}&=&\sum_{\gamma \in Q} \sum_{a\in \F_q} \gamma (x_1+ax_2)\otimes (ay_1-y_2) + \sum_{\gamma \in Q} \gamma x_2\otimes y_1\\
\end{eqnarray*}
By Lemma \ref{quadressum}, if $q\ne 3$, then $\sum_{\gamma \in Q}\gamma=0$ and the whole sum is zero, as claimed. For $q=3$, $Q=\{ 1\}$, so the sum is equal to
$$  \sum_{a\in\F_3} \left( -x_1 \otimes y_2-ax_2\otimes y_2+ax_1\otimes y_1+a^2x_2\otimes y_1 \right)+ x_2\otimes y_1=$$
$$=-x_2\otimes y_1+x_2\otimes y_1=0.$$
\end{proof}

\subsection{Description of $L_{1,c}(\triv)$ for $SL_n(\F_{p^r})$ if $n \ge 3$ or $p=2$}

As explained above and demonstrated in the case of $t=0$, we can get character formulas for $H_{1,c}(SL_n(\F_q),\h)$ directly from the ones for $H_{1,c}(GL_n(\F_q),\h)$ when $n\ge 3$ or $p=2$. The following is a corollary of Lemma \ref{normalsubgp}, Proposition \ref{sln-nosplit}, and results from Section \ref{glnsec}, most notably \ref{main}.

\begin{cor}\label{sl1}
Let $n \geq 3$ or $p = 2$. Consider the rational Cherednik algebra $H_{1,c}(SL_n(\F_q),\h)$, its representations $M_{1,c}(\mathrm{triv})$, the contravariant form $B$ on it, and the irreducible quotient $L_{1,c}(\mathrm{triv})$. Then all the results proved previously for the group $GL_n(\F_q)$ hold also for $SL_n(\F_q)$. Specifically, 
\begin{itemize}
\item[a)] $D_y(x^p) = 0$ in $M_{1,c}(\mathrm{triv})$. 
\item[b)] The form $B$ on $M_{1,c}(\mathrm{triv})$ is zero in degrees $np - n + 1$ and higher. 
\item[c)] The matrices of the form $B$ on $M_{1,c}(\mathrm{triv})$ in lexicographically ordered monomial bases are diagonal in all degrees.
\item[d)] All diagonal elements of the matrix of the form $B_i$ on any graded piece $M_{1,c}(\mathrm{triv})_i$ are $\Bbbk$-multiples of the same polynomial in $c$. 
\item[e)] If $(q,n) \ne (2,3)$, $(q,n) \ne (2,2)$ the matrices $B_i$ of the form on $M_{1,c}(\mathrm{triv})_i$ do not depend on $c$.
\item[f)] If $(q,n) = (2,3)$, only the matrices $B_2$ and $B_3$ depend on $c$. Their nonzero diagonal coefficients are constant multiples of $c+1$.
\item[g)] If $(q,n) = (2,2)$, then $GL_{n}(\F_q)=SL_{n}(\F_q)$ so the character formulas are the same.
\item[h)]The character formulas for $L_{1,c}(\triv)$ for the rational Cherednik algebra $H_{1,c}(SL_n(\F_q),\h)$ are the same as for the rational Cherednik algebra $H_{1,c}(GL_n(\F_q),\h)$, with the parameter function $c$ extended to all classes of reflections in $GL_n(\F_q)-SL_n(\F_q)$ by zero. 
\end{itemize}
\end{cor}

\subsection{Description of $L_{1,c}(\triv)$ for $SL_n(\F_{p^r})$ if $n = 2$ and $p > 2$}

As in the case of $t=0$, we need to study the case $n = 2$ and $p > 2$, when there are two conjugacy classes of reflections in $SL_n(\F_q)$, separately. The case $q=3$ is the most complicated and we solve it by calculating the matrices of the form $B$ explicitly. The following results address the remaining cases.

\begin{prop}\label{sln-n2-kernel}
Let $n=2$, $q=p^r$ for $p$ an odd prime, and $q\ne 3$. In the Verma module $M_{1,c}(\mathrm{triv})$ for $H_{1,c}(SL_{2}(\F_q))$, all the vectors $x^p, x\in \h^*$, are singular.
\end{prop}
\begin{proof}
We need to show that for every conjugacy class $C_T$, for $T=R$ or $T=Q$, and any $x,y$, the coefficient in $D_y(x^p)$ of $c_T$ is zero. This coefficient is equal to
$$\sum_{\alpha\otimes\alpha^{\vee}\in C_T}(y,\alpha)(x,\alpha^{\vee})^p\alpha^{p-1}.$$ Again, as $C_R\cup C_Q=C_1$ is a conjugacy class of reflections in $GL_2(\F_q)$, and the result holds there, it is enough to show this for $C_Q$. 

As in the proof of Proposition \ref{gln-kernel}, we claim that after writing it as a double sum, with the outer sum being over $\alpha$ and the inner over $\alpha^{\vee}$, the inner sum is zero for any $\alpha$. Fix $\alpha$ and change coordinates, so that we assume without loss of generality that $\alpha=x_1$. Then the inner sum is 
$$\sum_{\substack{\alpha^{\vee}=\gamma y_2 \\ \gamma \in Q}}(x,\alpha^{\vee})^p =(y_2,x)^{p}\sum_{\gamma\in Q}\gamma^p. $$

Using Lemma \ref{quadressum}, this is zero unless $p\equiv 0,\frac{q-1}{2} \pmod{q-1}$. However, this only happens in cases we excluded: $q=2$ and $q=3$.
\end{proof}

Next, we prove that acting by Dunkl operator produces elements of $M_{1,c}(\mathrm{triv})$ of a specific form. 

\begin{lemma}
Let $n=2$, $q=p^r$ for $p$ an odd prime, and $q\ne 3$, and consider the Verma module $M_{1,c}(\mathrm{triv})$ for $H_{1,c}(SL_{n}(\F_q))$. 
For any $f \in M_{1,c}(\mathrm{triv})/\left< x_i^p \right>\cong S\h^*/\left< x_i^p \right> $ and any $y \in \h$, there exists $h\in S\h^*$ such that, as elements of $M_{1,c}(\mathrm{triv})/\left< x_i^p \right>$,
$$D_y(f) =\partial_{y}f+  c_Q \cdot h+ c_R \cdot h $$
\end{lemma}
\begin{proof}
The Dunkl operator action in the case of $SL_{2}(\F_q)$ is
$$D_{y}(f)=\partial_{y}(f)-\sum_{T\in\{Q,R\} } c_{T}\sum_{s\in C_T}\frac{(\alpha_{s},y)}{\alpha_s}(1-s).f$$ The strategy is to compute the sum $\sum_{s\in C_T}\frac{(\alpha_{s},y)}{\alpha_s}(1-s).(f)$ parallelly  for $T=Q,R$, disregarding all terms that do not depend on the choice of $T$ (these contribute equally to the coefficient of $c_Q$ and $c_R$), and any elements of the ideal $\left< x_i^p \right>$. We will use:
\begin{eqnarray}
\sum_{a\in\F_q}a^m=0 & \textrm{unless} & m\equiv 0 \pmod{q-1}, m\ne 0 \label{1}\\
%\sum_{a\in\F_q}a^{q-1}&=&q-1  \\
\sum_{\gamma \in Q}\gamma ^m=\sum_{ \gamma\in R}\gamma^m & \textrm{unless} & m\equiv \frac{q-1}{2} \pmod{q-1} \label{2} \\
%\sum_{d\in Q}d^{m}=\sum_{d\in R}d^m & \textrm{if} & m\equiv \frac{q-1}{2} \pmod{q-1}\\
x_i^p=0 & \textrm{in} & M_{1,c}(\mathrm{triv})/\left< x_i^p \right> \label{3}
\end{eqnarray}
and the parametrization of conjugacy classes from \ref{sln-param}:
\begin{equation}
C_{T}=\left\{ \gamma \left[ \begin{array}{c} 1 \\ a \end{array} \right] \otimes \left[ \begin{array}{c} a \\ -1 \end{array} \right] \mid a \in \F_q, \, \gamma \in T\right\} \bigcup \left\{ \gamma \left[ \begin{array}{c} 0\\1 \end{array}\right] \otimes \left[ \begin{array}{c} 1 \\0 \end{array} \right] \mid \gamma \in T \right\} .
\label{4}
\end{equation}

The rest of the proof is this computation.

We will do it for $D_y(f)$, $y=y_1$, $f=x_1^ux_2^v$. The general statement follows from this case by symmetry and linearity. We can assume $u,v \leq p-1$. 

We claim that the sum
\begin{equation*} \tag{$\star$}
\sum_{s\in C_T}\frac{(\alpha_{s},y_1)}{\alpha_s}(x_1^ux_2^v-(s.x_1)^u(s.x_2)^v)
\end{equation*}
does not depend on $T$.

Reflections $s$ corresponding to elements of the form $\gamma \left[ \begin{array}{c} 0\\1 \end{array}\right] \otimes \left[ \begin{array}{c} 1 \\0 \end{array} \right] $ satisfy $(\alpha_{s},y_1)=(\gamma x_2,y_1)=0$, so they do not contribute to the sum. 

For $s$ of the form $\gamma \left[ \begin{array}{c} 1 \\ a \end{array} \right] \otimes \left[ \begin{array}{c} a \\ -1 \end{array} \right]$, let us write the action on $x_i\in\h^*$ explicitly, using  the notation $\alpha_a'=x_1+ax_2$. The explicit action is
\begin{eqnarray*}
s.x_1&=&x_1-a\gamma \alpha_a'\\
s.x_2&=&x_2+\gamma \alpha_a'.
\end{eqnarray*}
Substituting this into ($\star$), we get
\begin{eqnarray*}
(\star)&=&\sum_{\gamma \in T}\sum_{a\in\F_q}\frac{(\gamma \alpha_a',y_1)}{\gamma \alpha_a'}\left(x_1^ux_1^v -(x_1-a\gamma \alpha_a')^u(x_2+\gamma \alpha_a')^v\right)=\\
&=&\sum_{\gamma \in T}\sum_{a\in\F_q}\sum_{i=0}^{u}\sum_{\substack{j=0 \\(i,j)\ne(0,0)}}^{v}
\frac{1}{\alpha_a'}(-1){u \choose i}{v \choose j}x_1^{u-i}(-1)^{i}a^{i}\gamma^{i}\alpha_a'^i x_2^{v-j}\gamma^{j}\alpha_a'^j=\\
&=&\sum_{\gamma \in T}\sum_{a\in\F_q}\sum_{i=0}^{u}\sum_{\substack{j=0 \\(i,j)\ne(0,0)}}^{v}{u \choose i}{v \choose j}(-1)^{i+1}a^{i}\gamma^{i+j} x_1^{u-i}x_2^{v-j}(x_1+ax_2)^{i+j-1}=\\
&=&\sum_{\gamma \in T}\sum_{a\in\F_q}\sum_{i=0}^{u}\sum_{\substack{j=0 \\(i,j)\ne(0,0)}}^{v}\sum_{k=0}^{i+j-1}{u \choose i}{v \choose j}{i+j-1\choose k}(-1)^{i+1}a^{i+k}\gamma^{i+j} x_1^{u+j-1-k}x_2^{v-j+k}\\
\end{eqnarray*}

Let us evaluate the sum $\sum_{a\in \F_q} a^{i+k}$. By (\ref{1}), this is only nonzero if $i+k$ is divisible by $q-1$. We know that
$$i+k\le 2i+j-1\le2u+v-1\le3(p-1)-1<3(p-1)\le3(q-1),$$
so let us consider three different cases: $i+k=0$, $i+k=q-1$ and $i+k=2(q-1)$.

{\bf CASE 1: $i+k=0$.} In that case, $\sum_{a\in \F_q} a^{0}=0$, so this does not contribute either. 

{\bf CASE 2: $i+k=q-1$.} After substituting $\sum_{a\in \F_q} a^{q-1}=-1$, $k=q-1-i$, and after that $m=i+j$, we get that part of ($\star$) corresponding to this case equals
\begin{align*}
\sum_{\gamma\in T}\sum_{i=0}^{u}\sum_{\substack{j=0 \\(i,j)\ne(0,0)}}^{v}{u \choose i}{v \choose j}{i+j-1\choose q-1-i}(-1)^{i}\gamma^{i+j} x_1^{u+j+i-q}x_2^{v-j+q-1-i}=\\
=\sum_{\gamma\in T}\sum_{m=1}^{u+v}\sum_{j=0 }^{v}{u \choose m-j}{v \choose j}{m-1\choose q-1-m+j}(-1)^{m-j}\gamma^{m} x_1^{u+m-q}x_2^{v-m+q-1}
\end{align*}
Now we use (\ref{2}) to describe $\sum_{\gamma \in T}\gamma^m$ and disregard all terms except $m\equiv \frac{q-1}{2} \pmod{q-1}$ (these terms we disregarded contribute to the coefficient $h$). There are again few cases, as
$$m\le u+v\le 2(p-1)<\frac{5}{2}(p-1)\le \frac{5(q-1)}{2},$$
so we consider separately $m=\frac{q-1}{2}$ and $m=\frac{3(q-1)}{2}$.

{\bf CASE 2.1: $m=\frac{q-1}{2}$.} One of the binomial coefficients in the expression is ${m-1 \choose q-1-m+j}$, and we claim it is always zero for this choice of $m$. This is because
$$q-1-m+j=\frac{q}{2}-\frac{1}{2}+2j\ge \frac{q-1}{2}> \frac{q-1}{2}-1=m-1.$$
In other words, case $2.1.$ never actually appears in the sum. 

{\bf CASE 2.2: $m=\frac{3(q-1)}{2}$.} We will show that this part of the sum is zero. First, 
$$\frac{3(q-1)}{2}=m\le u+v\le2(p-1)$$ implies that this can only happen when $p=q$. Next, because we $x_1^p=0$ in the quotient, the only terms that can be nonzero are the ones with the power of $x_1$ less than $p$, so
$$u+m-q\le p-1$$ which means
$$u\le p-1+q-m=\frac{p+1}{2}.$$ Next, the term ${u \choose m-j}$ is zero unless
$$j\ge m-u\ge \frac{3(p-1)}{2}-\frac{p+1}{2}=p-2.$$
Since $$j\le v\le p-1,$$ it follows that $j\in\{ p-2,p-1 \}$. In both those cases, the binomial coefficient ${m-1 \choose q-1-m+j}$ is $0$, as the numerator has a factor $p$ and the denominator does not.

{\bf CASE 3: $i+k=2(q-1)$.} The part of the sum ($\star$) corresponding to this case is
\begin{align*}
\sum_{\gamma \in T}\sum_{i=0}^{u}\sum_{\substack{j=0 \\(i,j)\ne(0,0)}}^{v}{u \choose i}{v \choose j}{i+j-1\choose 2q-2-i}(-1)^{i}\gamma^{i+j} x_1^{u+j-1-2q+2+i}x_2^{v-j+2q-2-i}=\\
\sum_{\gamma \in T}\sum_{m=1}^{u+v}\sum_{j=0 }^{v}{u \choose m-j}{v \choose j}{m-1\choose 2q-2-m+j}(-1)^{m-j}\gamma^{m} x_1^{u+m-1-2q+2}x_2^{v-m+2q-2}
\end{align*}
The powers of $x_2$ in this sum and the original power of $x_1$ are both $\le p-1$, so
$$p-1\ge v-m+2(q-1)$$
$$p-1\ge u\ge m-v\ge 2(q-1)-(p-1)\ge p-1.$$
From the last string of inequalities, $u=p-1$, $m=u+v$ and $p=q$. The above sum then becomes
\begin{align*}
\sum_{\gamma \in T}\sum_{j=0 }^{v}{p-1 \choose p-1+v-j}{v \choose j}{p+v-2\choose p-1-v+j}(-1)^{p-1+v-j}\gamma^{p-1+v} x_1^{v-1}x_2^{p-1}
\end{align*}
As $j\le v$, the first binomial coefficient in this sum is zero unless $j=v$, producing
\begin{align*}
\sum_{\gamma \in T}{p+v-2\choose p-1}\gamma^{p-1+v} x_1^{v-1}x_2^{p-1}.
\end{align*}
The sum $\sum_{\gamma \in T}\gamma^{p-1+v}=\sum_{\gamma \in T}\gamma^{v}$ only depends on $T$ if $v\equiv \frac{p-1}{2} \pmod{p-1},$ which only happens if $v=\frac{p-1}{2}$. In that case, ${p+v-2\choose p-1}={p+v-2\choose p-1}=0$, as the numerator is divisible by $p$.
\end{proof}

%\begin{prop}\label{sln-n2-coeffeq}
%Assuming $n=2$, $q=p^r$ for $p$ an odd prime, and $q\ne 3$, in the Verma module $L_{1,c}(\mathrm{triv})$ for $H_{1,c}(SL_{n}(\F_q))$, 
%The coefficients of $c_1^i c_2^j$ and $c_1^j c_2^i$ are equal in the matrices of $B$.
%\end{prop}
%\begin{proof}
%We prove this inductively on the degree of $B$.  The base case is trivial.  Now, $B(x_1^a x_2^b, y_1^c y_2^d) = B(x_1^{a-1} x_2^b, D_1(y_1^c y_2^)) = B(x_1^{a-1} x_2^b, h(c_1) + h(c_2) + h_0 + h'(c_1, c_2)) = B(x_1^{a-1} x_2^b, h(c_1) + h(c_2) + h_0)$.  By induction, the coefficients of $c_1^i c_2^j$ and $c_1^j c_2^i$ are equal in this expression, and the proposition follows.
%\end{proof}

We can use the previous proposition to transfer the results we had about $GL_{2}(\F_q)$ to $SL_{2}(\F_{q})$, as in the previous chapter. Namely, the structure of irreducible modules for $H_{1,c}(SL_{2}(\F_{q}),\h)$, where $c$ takes values $c_Q$ on $C_Q$ and $c_R$ on $C_R$, is determined by Dunkl operators. By the previous proposition,
$$\sum_{s\in C_{Q}}\frac{(\alpha_s,y)}{\alpha_s}(1-s)=\sum_{s\in C_{R}}\frac{(\alpha_s,y)}{\alpha_s}(1-s),$$ so the Dunkl operator is equal to
\begin{eqnarray*}D_{y}&=&\partial_{y}-\sum_{s\in C_{Q}}c_{Q}\frac{(\alpha_s,y)}{\alpha_s}(1-s)-\sum_{s\in C_{R}}c_{R}\frac{(\alpha_s,y)}{\alpha_s}(1-s)\\
%&=&\partial_{y}-(c_{R}+c_{Q})\sum_{s\in C_{Q}}\frac{(\alpha_s,y)}{\alpha_s}(1-s)\\
%&=&\partial_{y}-(c_{R}+c_{Q})\sum_{s\in C_{R}}\frac{(\alpha_s,y)}{\alpha_s}(1-s)\\
&=&\partial_{y}-\frac{c_{R}+c_{Q}}{2}\sum_{s\in C_{R}\cup C_{Q}}\frac{(\alpha_s,y)}{\alpha_s}(1-s)\\
\end{eqnarray*}
In $GL_{2}(\F_q)$, the union $C_{Q}\cup C_{R}$ is one conjugacy class $C_{1}$ (unipotent reflections). Define the function $c$ on all reflections in $GL_{2}(\F_q)$ by letting it be $c_1= \frac{c_{R}+c_{Q}}{2}$ on all unipotent reflections, and $c_{\lambda}=0$ on all semisimple reflections. Then the Dunkl operators controlling the structure of $L_{1,c}(\mathrm{triv})$ for $H_{1,c}(GL_2(\F_q),\h)$ are 
\begin{eqnarray*}
D_{y}&=&\partial_{y}-\sum_{\lambda=1}^{q}\sum_{s\in C_{\lambda}}c_{\lambda}\frac{(\alpha_s,y)}{\alpha_s}(1-s)\\
&=&\partial_{y}-\sum_{s\in C_{1}}c_{1}\frac{(\alpha_s,y)}{\alpha_s}(1-s),
\end{eqnarray*}
which is exactly the same as the Dunkl operator for $H_{1,c}(SL_{2}(\F_{q}),\h)$. From this  we get:

\begin{cor}
Let $n=2$, $q=p^r$ for $p$ an odd prime, and $q\ne 3$, and consider the Verma module $L_{1,c}(\mathrm{triv})$ for $H_{1,c}(SL_{n}(\F_q))$. All the results proved previously for the rational Cherednik algebra associated to $GL_n(\F_q)$ hold for $SL_n(\F_q)$. Namely,
\begin{itemize}
\item[a)] $D_y(x^p) = 0$ in $M_{1,c}(\mathrm{triv})$. 
\item[b)] The form $B$ on $M_{1,c}(\mathrm{triv})$ is zero in degrees $2p - 1$ and higher. 
\item[c)] The form $B$ on $M_{1,c}(\mathrm{triv})$ is diagonal in all degrees. 
\item[d)] All diagonal elements of the matrix of the form $B_i$ on any graded piece $M_{1,c}(\mathrm{triv})_i$ are $\Bbbk$-multiples of a single polynomial in $c$. 
\item[e)]  If $q = p^r$ with $r > 1$, then $B$ does not depend on $c$.
\item[f)]  If $q = p$, the matrices of $B_i$ on $M_{1,c}(\mathrm{triv})_i$ are constant for $i=0,\ldots , p-2$, constant multiples of $c_Q + c_R - 2$ for $i=p-1$, and constant multiples of $(c_Q + c_R - 2)(c_Q + c_R + 2)$ for $i=p,\ldots , 2p-2$.
\end{itemize}
\end{cor}

Putting together the previous Corollary, Corollary \ref{sl1}, explicit computations for the rational Cherednik algebra associated to $SL_{2}(\F_3)$, and noticing that  $SL_2(\F_2) = GL_2(\F_2)$, we get the main theorem of this section.

\begin{thm}\label{sln}
Let $p$ be a prime, $q = p^r$ and $n\ge 2$.  The characters of $L_{1,c}(\mathrm{triv})$ for the rational Cherednik algebra $H_{1,c}(SL_{n}(\F_{q}),\h)$ over an algebraically closed field of characteristic $p$ are as follows:\\

\begin{tabular}{|c|c|c|c|}
\hline
$(q,n)$ & $c$ & $\chi_{L_{1,c}(\triv)}(z)$ & $\mathrm{Hilb}_{L_{1,c}(\triv)}(z)$ \\ \hline \hline
$(q,n)\ne (2,2),(3,2)$& generic & $\sum_{i\ge 0}(\left[ S^i\h^*/\left<x_1^p,\ldots , x_n^p \right>\cap S^i\h^*\right])z^i$ & $\left(\displaystyle\frac{1 - z^p}{1 - z}\right)^n$ \\ \hline
$(3,2)$ & generic &$\left(\sum_{i\ge 0}\left[S^i\h^*\right] z^i\right)(1-z^{12})(1-z^{18})$ & $\displaystyle\frac{(1 - z^{12})(1 - z^{18})}{(1 - z)^2}$ \\ \hline
$(2,2)$ & generic & $\left(\sum_{i\ge 0}\left[S^i\h^*\right] z^i\right)(1-z^4)(1-z^6)$ & $\displaystyle\frac{(1 - z^4)(1 - z^6)}{(1 - z)^2}$ \\ \hline
\hline
%$(q,n)$ & $c$ & $\chi_{L_{1,c}(\tau)}(z)$ & $\mathrm{Hilb}_{L_{1,c}(\tau)}(z)$ \\ \hline
$(2,2)$ & $c=1$ & $\displaystyle{\left[\mathrm{triv}\right]}$ & $1$ \\ \hline
$ (2,2)$ & $c=0$ & $\displaystyle{\left[\mathrm{triv}\right]+\left[\h^*\right]z+\left[\mathrm{triv}\right]z^2}$ & $1+2z+z^2$ \\ \hline
$ (2,3)$ & $c=1$ & $\displaystyle{\left[\mathrm{triv}\right]+\left[\h^*\right]z}$ & $1+3z$ \\ \hline
$(p, 2), p \ne 2,3$ &  $c_Q + c_{R} = 2$ & $\displaystyle{\sum_{i=0}^{p-2}\left[ S^i\h^*\right]z^i}$& $\displaystyle{\sum_{i=0}^{p-2}(i+1)z^i } $  \\ \hline
$(p, 2), p \ne 2,3$ & $c_Q + c_{R} = -2$& $\displaystyle{\sum_{i=0}^{p-1}\left[ S^i\h^*\right]z^i}$ & $\displaystyle{\sum_{i=0}^{p-1}(i+1)z^i } $ \\ \hline
\end{tabular}

We omit the characters for $(q, n) = (3,2)$ and special $c$ as there are too many cases to concisely list.

\end{thm}

\begin{rem}
All $L_{t,c}(\triv)$ for generic $c$ have one dimensional top degree and are Frobenius. For $h_1$ the reduced Hilbert series of $L_{1,c}(\triv)$ at generic $c$, $h_1(1)$ is either $|G|$ or $1$. For $h_0$ the Hilbert series of $L_{0,c}(\triv)$ at generic $c$, the inequality $h_0\le h_1$ term by term always holds, but not the equality: for $SL_2(\F_3)$, $h_0(z)=1$, and $h_1(z)=\frac{(1-z^4)(1-z^6)}{(1-z)^2}$.
\end{rem}

%\section{Conjectures and Observations}\label{conjsec}

%\noindent Let $G = O_2(\F_q)$ be the subgroup of $GL_2(\F_q)$ invariant with respect to some quadratic form.  If the quadratic form is split, let $d = q-1$.  If it is nonsplit, let $d = q+1$.
%\begin{conj}
%Let $G = O_2(\F_q)$.  The reduced character of $L_{1,c}(\triv)$ is $h(t) = \frac{(1 - t^{2})(1 - t^{d})}{(1 - t)^2}$.
%\end{conj}

%\noindent If $p \nmid |G|$, then $G$ is a reflection group in characteristic zero.
%\begin{conj}
%For $G = S_n \ltimes (\Z_{l})^n$ and $p \nmid |G|$, $L_{1,c}(\triv)$ has reduced character $h(t) = \prod_{j=1}^n \frac{1 - t^{jl}}{1 - t}$ and dimension $|G|p^n$.
%\end{conj}

%\begin{obs}
%For generic $c$, the reduced character for $L_{1,c}(\triv)$ is often $h(t) = \prod_{i=1}^n \frac{1 - t^{d_i}}{1 - t}$ where $d_1, \ldots, d_n$ are the degrees for the generating $G$-invariants in $S\h^*$.  Under what conditions is this true?
%\end{obs}
%Indeed, the observation cannot always be true, as this would imply that $h(1) = \prod d_i = $ ???????????//

%\begin{obs}
%Recall Corollary \ref{lcdim}, which stated that $L_{1,c}(\tau)$ has dimension $h(1) p^n$, and $1 \leq h(1) \leq |G|$.  Most examples we have considered have $h(1) = |G|$ or $h(1) = 1$.  We found one instance where $1 < h(1) <|G|$.  If $G = S_4$ over $\F_3$, the reduced generic character of $L_{1,c}(\triv)$ is $h(t) = \frac{(1 - t^2)^2(1 - t^3)}{(1 - t)^3}$.  In this case, $h(1) = 12$.  So far, in all instances we have observed, $h(1) \mid |G| = 24$.
%\end{obs}

%%%%%%%%%%%%%%%%%%%%%%%%%%%%%%%%%%%%%

\clearpage
\appendix

\section{A conjecture about orthogonal groups}\label{secdata}
In the preliminary stages of this project, we used MAGMA \cite{magma} to gather data about characters of irreducible representations of a larger class of groups than we ended up studying. Here we list some conjectures about Hilbert series of irreducible representations with trivial lowest weight for orthogonal groups over a finite field. 

%Two forms defined by matrices $A$ and $A'$ are equivalent if there exists an invertible matrix $X$ such that $A'=X^{t}AX$; in that case the groups $O_{n}(A,\F_q)$ and $O_{n}(A',\F_q)$ are isomorphic. Additionally, the groups $O_{n}(A,\F_q)$ and $O_{n}(\lambda A,\F_q)$, $\lambda\in \F_q^{\times}$, are isomorphic. 

The groups $O_{n}(A,\F_q)$ are defined as subgroups of $GL_n(\F_q)$ preserving a quadratic form $\F_q^n\to \F_q$ given by an invertible matrix $A$ as $x\mapsto x^{t}Ax$. For odd $n$ there is only one orthogonal group up to isomorphism, so we choose $A=I$ and denote such a group by $O_n(\F_q)$. For even $n$, there are two possible nonisomorphic groups: $O_n^{+}(\F_q)$, preserving the form given by $A$ a diagonal matrix with diagonal entries $1,-1,1,-1,\ldots 1,-1$, (so that the total space $\Bbbk^n$ is an orthogonal sum of hyperbolic lines $((x_1,x_2),(y_1,y_2))\mapsto x_1y_1-x_2y_2$) and  $O_n^{-}(\F_q)$, preserving the form given by $A$ a diagonal matrix with diagonal entries $1,-1,1,-1,\ldots 1,-1, 1, -r$, for $r$ an arbitrary quadratic non-residue (in which case the total space $\Bbbk^n$ is an orthogonal sum of several hyperbolic lines and a two dimensional anisotropic space). 

\begin{conj}
The reduced Hilbert series for irreducible representation with trivial lowest weight of rational Cherednik algebras associated to orthogonal groups of low rank are as follows:
\begin{eqnarray*}
\mathrm{Hilb}_{O_2^+(\F_q)}(z)&=&(1+z)(1+z+z^2+\ldots z^{q-2})\\
\mathrm{Hilb}_{O_2^-(\F_q)}(z)&=&(1+z)(1+z+z^2+\ldots z^{q})\\
\mathrm{Hilb}_{O_3(\F_q)}(z)&=&1.
\end{eqnarray*}
\end{conj}

The conjecture was checked computationally for pairs $(n,q)=$ $(2,3)$, $(2,5)$, $(2,9)$, $(3,2)$ and $(3,3)$.

%Remember that the Hilbert series of $L_{1,c}(\mathrm{triv})$ for generic $c$ is of the form $h(z^p)\left( \frac{1-z^p}{1-z}\right)$, where $n$ is the dimension of the refection representation and $p$ is the characteristic of the underlying field. The reduced character $h(z)$ for various reflection groups $G$ is as follows:\\

%$G = O_2(\F_3)$ split, $h(z)=(1+z)^2$\\
%$G = O_2(\F_3)$ nonsplit, $h(z)=(1+z)(1+z+z^2+z^3)$ \\
%$G = O_2(\F_5)$ split, $h(z)=(1+z)(1+z+z^2+z^3)$\\
%$G = O_2(\F_5)$ nonsplit, $h(z)=(1+z)(1+z+z^2+z^3+z^4+z^5)$ \\
%$G = O_2(\F_9)$ split, $h(z)=(1+z)(1+z+z^2+z^3+z^4+z^5+z^6+z^7)$  \\
%$G = O_2(\F_9)$ nonsplit, $h(z)=(1+z)(1+z+z^2+z^3+z^4+z^5+z^6+z^7+z^8+z^9)$ \\
%$G = O_3(\F_2)$, $h(z)=1$\\
%$G = O_3(\F_3)$,  $h(z)=1$\\
%$G = S_3(\F_2)$, $h(z)=1+2z+2z^2+z^3$ \\
%$G = S_4(\F_2)$, $h(z)=1$ \\
%$G = S_5(\F_2)$, $h(z)=1+4z+4z^2+z^3$ \\
%$G = S_4(\F_3)$, $h(z)=1+3z+4z^2+3z^3+z^4$ \\
%$G = S_3(\F_3)$, $h(z)=1$ \\
%$G = S_3(\F_5)$, $h(z)=1+2z+2z^2+z^3$  \\

%\noindent In the following, our computations for $\h^*$ likely contain an error. \\
%$G = GL_2(\F_2)$, $\tau = \h$, $\chi =  [ 1, 4, 4, 4, 4, 4, 2, 1 ]$ \\
%$G = GL_3(\F_2)$, $\tau = \h$, $\chi = [ 1, 9, 12, 12, 9, 3 ]$ \\
%\noindent In the following, our computations yielded the same results for $\tau = \h$ and $\tau = \h^*$. \\
%$G = GL_4(\F_2)$, $\tau = \h$, $\chi = [ 1, 16, 24, 16, 4 ]$ \\
%$G = GL_2(\F_3)$, $\tau = \h$, $\chi = [ 1, 4, 6, 7, 8, 9, 8, 7, 6, 4, 2 ]$ \\
%$G = GL_3(\F_3)$, $\tau = \h$, $\chi = [ 1, 9, 18, 21, 18, 9, 3 ]$ \\
%\end{data}

\section*{Acknowledgments}
We are very grateful to Pavel Etingof for suggesting the problem and devoting his time to it through numerous helpful conversations. This project was initiated as part of the Summer Program for Undergraduate Research (SPUR) at the Department of Mathematics at MIT, and partially funded by SPUR and Undergraduate Research Opportunities Program (UROP) at MIT. The work of H.C. was partially supported by the Lord Foundation through UROP. The work of M.B. was partially supported by the NSF grant  DMS-0758262.

%\section*{References}
\bibliographystyle{plain}
\bibliography{sources}

\begin{thebibliography}{1}

\bibitem{artin}
Michael Artin.
\newblock Noncommutative rings.
\newblock \url{http://www-math.mit.edu/~etingof/artinnotes.pdf}, 1999.

\bibitem{magma}
Wieb Bosma, John Cannon, and Catherine Playoust.
\newblock {The Magma algebra system}.
\newblock {\em Journal of Symbolic Computation}, 24(3-4):235--265, 1997.

\bibitem{dickson}
Leonard~Eugene Dickson.
\newblock {A fundamental system of invariants of the general modular group with
  a solution of the form problem}.
\newblock {\em Transactions of the American Mathematical Society},
  12(1):75--98, 1911.

\bibitem{etingof-ma}
Pavel Etingof and Xiaoguang Ma.
\newblock {Lecture notes on Cherednik algebras}.
\newblock {\em arXiv}, 1001.0432v4, 2010.

\bibitem{gordon}
Iain Gordon.
\newblock {Baby Verma modules for rational Cherednik algebras}.
\newblock {\em Bulletin of the London Mathematics Society}, 35(3):321--336,
  2003.

\bibitem{griffeth}
Stephen Griffeth.
\newblock {Towards a combinatorial representation theory for the rational
  Cherednik algebra of type $G(r,p,n)$}.
\newblock {\em Proceedings of the Edinburgh Mathematical Society},
  53(2):419--445, 2010.

\bibitem{kemper-malle}
Gregor Kemper and Gunter Malle.
\newblock The finite irreducible linear groups with polynomial ring of
  invariants, transformation groups 2.
\newblock {\em Transformation Groups}, (2):57--89, 1997.

\bibitem{Krop}
Leonid Krop.
\newblock On the representations of the full matrix semigroup on homogeneous
  polynomials.
\newblock {\em Journal of Algebra}, 99(2):370 -- 421, 1986.

\bibitem{smithl}
Larry Smith.
\newblock {\em {Polynomial Invariants of Finite Groups}}.
\newblock A K Peters, Ltd., Wellesley, Massachusetts, 1995.

\end{thebibliography}

\end{document}